\documentclass[a4paper,10pt]{article}
\usepackage[utf8]{inputenc}
\usepackage{lmodern} 
\usepackage[T1]{fontenc}
\usepackage{textcomp}
\usepackage[scr=boondoxo,frak=boondox,bb=boondox]{mathalfa}
\title{\normalsize{\textbf{INNER AND PARTIAL NON-DEGENERACY OF MIXED FUNCTIONS}}}
\author{Benjamin Bode \and Eder L. Sanchez Quiceno}
\newcommand{\Addresses}{{
  \bigskip
  \footnotesize

B.~Bode, \textsc{Instituto de Ciencias Matemáticas (ICMAT), Consejo Superior de Investigaciones Científicas (CSIC), Campus Cantoblanco UAM,\\
C/ Nicolás Cabrera, 13-15, 28049 Madrid, Spain}\par\nopagebreak
  \textit{E-mail address}: \texttt{benjamin.bode@icmat.es}

  \medskip

E.~L.~Sanchez Quiceno, \textsc{Institute of Mathematics and Computer Science (ICMC)\\
University of São Paulo (USP), Avenida Trabalhador São-Carlense, 400 - Centro\\
CEP: 13566-590 - São Carlos - SP, Brazil}\par\nopagebreak
  \textit{E-mail address}: \texttt{ederleansanchez@usp.br}

}}
\date{}
\usepackage{graphicx}
\usepackage[all]{xy}
\usepackage{mathrsfs}
\usepackage[usenames,dvipsnames]{color}
\usepackage{makeidx}
\usepackage{subcaption}
\usepackage{graphicx}
\usepackage{float}
\usepackage{indentfirst}
\usepackage{amsmath}
\usepackage{amsthm}
\usepackage{graphicx,color}
\usepackage{setspace}
\usepackage{epigraph}
\usepackage{lmodern}
\usepackage{amsfonts,amssymb,amsmath,latexsym}
\usepackage{verbatim}
\usepackage{lscape}
\usepackage{tikz-cd}
\usetikzlibrary{calc,intersections,through,backgrounds}
\usepackage{xfrac}    
\usepackage{pgfplots}
\usepackage{mathtools}
\usepackage{tikz}
\usepackage{faktor}
\usepackage{amsmath}
\usepackage{braids}
\usepackage{pinlabel}
\usepackage{titlesec}
\usepackage{lipsum}
\usepackage{soul}
\titleformat{\section}
        [block]
        {\small \mdseries\centering}
        {\thesection}
        {1ex}
        {}
        [
        ]%
        \titleformat{\subsection} [block] {\large \mdseries}
        {\thesubsection} {1ex} {}
        [
        ] 
\usepackage[english]{babel}
\usepackage{hyperref}
\hypersetup{
    colorlinks=true,
    linkcolor=blue,
    filecolor=magenta,      
    urlcolor=cyan,
}
\makeindex
\newtheorem{teo}{Theorem}[section]
\newtheorem{corolario}[teo]{Corollary}
\newtheorem{prop}[teo]{Proposition}
\newtheorem{lemma}[teo]{Lemma}
\newtheorem{definition}[teo]{Definition}
\newtheorem{ex}[teo]{Example}
\newtheorem{obs}[teo]{Remark}

\DeclareMathOperator{\re}{Re}
\DeclareMathOperator{\im}{Im}

\newcommand{\C}{\mathbb{C}}       
\newcommand{\R}{\mathbb{R}}       
\newcommand{\N}{\mathbb{N}}       

\newcommand{\rme}{\mathrm{e}}
\newcommand{\rmi}{\mathrm{i}}
\newcommand{\rmd}{\mathrm{d}}
\newcommand{\defeq}{\mathrel{\mathop:}=}
   \makeatletter
\@namedef{subjclassname@2020}{%
  \textup{2020} Mathematics Subject Classification}
\begin{document}
\maketitle
\begin{abstract}
Mixed polynomials $f:\C^2\to\C$ are polynomial maps in complex variables $u$ and $v$ as well as their complex conjugates $\bar{u}$ and $\bar{v}$. They are therefore identical to the set of real polynomial maps from $\mathbb{R}^4$ to $\mathbb{R}^2$. We generalize Mondal's notion of partial non-degeneracy from holomorphic polynomials to mixed polynomials, introducing the concepts of partially non-degenerate and strongly partially non-degenerate mixed functions. We prove that partial non-degeneracy implies the existence of a weakly isolated singularity, while strong partial non-degeneracy implies an isolated singularity. We also compare (strong) partial non-degeneracy with other types of non-degeneracy of mixed functions, such as (strong) inner non-degeneracy, and find that, in contrast to the holomorphic setting, the different properties are not equivalent for mixed polynomials. We then introduce additional conditions under which strong partial non-degeneracy becomes equivalent to the existence of an isolated singularity. Furthermore, we prove that mixed polynomials that are strongly inner non-degenerate satisfy the strong Milnor condition, resulting in an explicit Milnor (sphere) fibration.

    \vspace{0.2cm} 

   
     \vspace{0.2cm} 
     
   
\end{abstract}

\section{INTRODUCTION}

Singularities of holomorphic polynomials are quite well studied. In particular, there are several non-degeneracy conditions that imply isolatedness of singularities and that lead to invariants associated with the Newton boundary that allow insights into topological properties of the singularity. In this paper we study real polynomial maps or, equivalently in the dimensions that we consider, mixed polynomials. We introduce the notion of a partially Newton non-degenerate mixed function, generalizing the corresponding definition from the complex setting, and compare it with other non-degeneracy conditions of mixed functions, in particular with those of Oka \cite{Oka2010} and our previous work with Araújo dos Santos \cite{AraujoBodeSanchez}.  

Complex polynomials that satisfy certain non-degeneracy conditions are known to be accessible to a topological analysis. For a special class of polynomials defined by Kouchnirenko \cite{Kouchnirenko1976} for example, which are called non-degenerate with respect to their Newton boundary (or Kouchnirenko non-degenerate), the topological type of the singularity, i.e., the link type of the singularity, is determined by the terms on the Newton boundary, that is, its \textit{Newton principal part} \cite{king1978, FukuiYoshinaga1985}. Moreover, several singularity invariants are determined by combinatorial aspects of the Newton boundary, for instance, the formula of the Milnor number for a convenient, non-degenerate holomorphic polynomial presented by Kouchnirenko \cite{Kouchnirenko1976}. Wall \cite{wall} defined inner non-degenerate polynomials, generalizing Kouchnirenko’s formula to include weighted homogeneous holomorphic polynomials with isolated singularity. Later, Mondal generalized Wall’s non-degeneracy, defining the concept of a partially non-degenerate polynomial \cite{mondal}. These two notions are known to be different in a field of positive characteristic, but are conjectured to be equivalent in fields of characteristic zero. The conjecture was proved for polynomials in 2 or 3 variables by Mondal. 

A similar development has occurred over the last years in the study of real polynomial mappings. We consider polynomial mappings $f:\R^4 \to \R^2$, with $f(O)=0$, where $O$ denotes the origin in $\R^4$. These are the natural real analogues of holomorphic polynomials $\C^2 \to \C$. In order to take better advantage of this analogy we may write $f$ as a mixed polynomial, a complex-valued function in two complex variables and their conjugates. The only reason for restricting our attention to mixed polynomials of these dimensions is that our original motivation was the study of the (classical 1-dimensional) links of singularities, see also \cite{AraujoBodeSanchez}. We expect that our definitions and results have analogues for any mixed polynomial $f:\C^k\to \C$. The set of such mixed polynomials is of course identical to the set of real polynomial mappings of the appropriate dimensions.

The singular set $\Sigma_f$ for a mixed polynomial, which is the set of points where the real Jacobian matrix does not have maximal rank, can be equivalently defined as the set of solutions to a system of equations of mixed polynomials, to be described in more detail in the later sections. We say that the origin $O\in\mathbb{R}^4$ is a weakly isolated singularity if it is a singularity and there is an open neighbourhood $U \subset \R^4$ of the origin such that $U \cap V_f \cap \Sigma_f = \{O\}$, where $V_f$ is the variety defined by the equation $f = 0$. We say that $O$ is an isolated singularity if it is a singularity and there is a neighbourhood $U$ of the origin such that $U\cap \Sigma_f =\{O\}$. The difference between weakly isolated singularities and isolated singularities marks an important departure from the complex case, where the two notions are equivalent.

The concepts of a Newton polygon and Newton boundary of holomorphic polynomials have an analogous definition for mixed functions as introduced by Oka \cite{Oka2010}. Following this analogy with the holomorphic setting, different non-degeneracy conditions have been put forward. Oka introduced the notion of Newton non-degeneracy and strong Newton non-degeneracy, which imply the existence of a weakly isolated singularity and an isolated singularity, respectively, if another property (``convenience'') of the Newton boundary is assumed.

We defined the concept of inner non-degenerate and strongly inner non-degenerate mixed polynomials in \cite{AraujoBodeSanchez}. Again these non-degeneracy conditions only depend on the terms on the Newton boundary. It was shown that inner non-degeneracy and strongly inner non-degeneracy imply a weakly isolated singularity and an isolated singularity, respectively. Furthermore, Oka's convenient and (strongly) non-degenerate mixed functions were found to be special cases of our (strongly) inner non-degenerate functions. 

The topological type of an inner non-degenerate polynomial with an additional ``nice'' property is determined by the terms on the mixed Newton boundary. An assumption of an isolated singularity is not necessary to prove this equivalence, as opposed to the results in \cite{oka2018}. Therefore, these classes of mixed polynomials are interesting in that they possess certain properties that are already established for holomorphic polynomials. Of particular interest in this context is the family of semiholomorphic polynomials that was introduced in \cite{bode_lemniscate} and studied in detail in \cite{AraujoBodeSanchez}. This family consists of those mixed polynomials that are holomorphic with respect to one of the two complex variables.

In this paper, we define partial Newton non-degeneracy and strong partial Newton non-degeneracy (Definition~\ref{Newtoncond} and Definition~\ref{Newtoncondstrong}), which are mixed versions of the partial non-degeneracy introduced by Mondal. As with other types of non-degeneracy we often drop ``Newton'' and simply say that a polynomial is partially non-degenerate. We prove that partial Newton non-degeneracy~(PND) and strong partial Newton non-degeneracy~(SPND) generalize previous conditions that imply weakly isolated singularity and isolated singularity, that are, convenience and non-degeneracy~(ND), and convenience and strong non-degeneracy~(SND) defined in \cite{Oka2010} and inner non-degeneracy~(IND) and strong inner non-degeneracy (SIND) defined in \cite{AraujoBodeSanchez}.

\begin{teo}\label{lem:innerpar}
Let $f$ be a (strongly) inner non-degenerate mixed polynomial. Then $f$ is (strongly) partially non-degenerate.
\end{teo}

\begin{teo}\label{strong-isolated3}
Let $f$ be a partially non-degenerate mixed polynomial. Then $f$ has a weakly isolated singularity at the origin. If $f$ is strongly partially non-degenerate, then it has an isolated singularity at the origin.  
\end{teo}

We can see the relations between these different types of non-degeneracy and the existence of isolated singularities in the following diagrams.

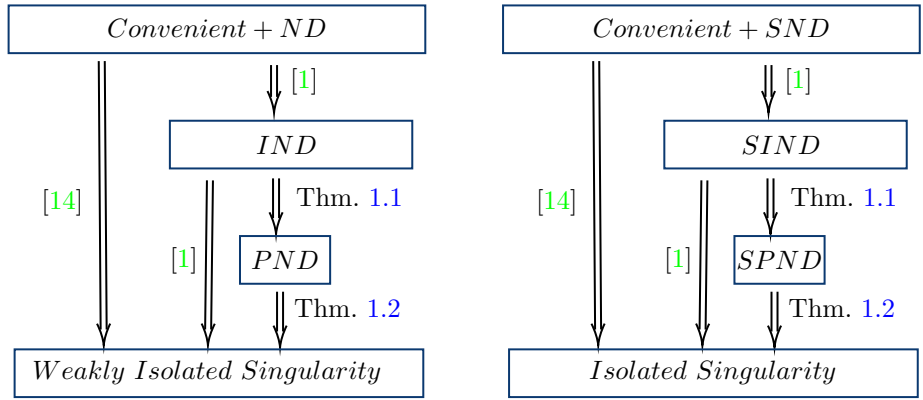
\begin{figure}[H]
\begin{subfigure}[a]{.4\textwidth}
\tikzset{every picture/.style={line width=0.75pt}} 

\begin{tikzpicture}[x=0.75pt,y=0.75pt,yscale=-1,xscale=1]

\draw [line width=0.75]    (135,36) -- (135,52)(132,36) -- (132,52) ;
\draw [shift={(133.5,60)}, rotate = 270] [color={rgb, 255:red, 0; green, 0; blue, 0 }  ][line width=0.75]    (10.93,-3.29) .. controls (6.95,-1.4) and (3.31,-0.3) .. (0,0) .. controls (3.31,0.3) and (6.95,1.4) .. (10.93,3.29)   ;
\draw [line width=0.75]    (49,33.99) -- (49.94,168.99)(46,34.01) -- (46.94,169.01) ;
\draw [shift={(48.5,177)}, rotate = 269.6] [color={rgb, 255:red, 0; green, 0; blue, 0 }  ][line width=0.75]    (10.93,-3.29) .. controls (6.95,-1.4) and (3.31,-0.3) .. (0,0) .. controls (3.31,0.3) and (6.95,1.4) .. (10.93,3.29)   ;
\draw [line width=0.75]    (103,94.02) -- (102.1,169.02)(100,93.98) -- (99.1,168.98) ;
\draw [shift={(100.5,177)}, rotate = 270.69] [color={rgb, 255:red, 0; green, 0; blue, 0 }  ][line width=0.75]    (10.93,-3.29) .. controls (6.95,-1.4) and (3.31,-0.3) .. (0,0) .. controls (3.31,0.3) and (6.95,1.4) .. (10.93,3.29)   ;
\draw [line width=0.75]    (136,93) -- (136,113)(133,93) -- (133,113) ;
\draw [shift={(134.5,121)}, rotate = 270] [color={rgb, 255:red, 0; green, 0; blue, 0 }  ][line width=0.75]    (10.93,-3.29) .. controls (6.95,-1.4) and (3.31,-0.3) .. (0,0) .. controls (3.31,0.3) and (6.95,1.4) .. (10.93,3.29)   ;
\draw [line width=0.75]    (138,150) -- (138,170)(135,150) -- (135,170) ;
\draw [shift={(136.5,178)}, rotate = 270] [color={rgb, 255:red, 0; green, 0; blue, 0 }  ][line width=0.75]    (10.93,-3.29) .. controls (6.95,-1.4) and (3.31,-0.3) .. (0,0) .. controls (3.31,0.3) and (6.95,1.4) .. (10.93,3.29)   ;

\draw  [color={rgb, 255:red, 11; green, 57; blue, 112 }  ,draw opacity=1 ]  (1,6) -- (209,6) -- (209,30) -- (1,30) -- cycle  ;
\draw (105,18) node    {$\ \ \ \ \ \ \ \ \ Convenient+ND\ \ \ \ \ \ \ \ \ $};
\draw  [color={rgb, 255:red, 11; green, 57; blue, 112 }  ,draw opacity=1 ]  (81.5,64) -- (202.5,64) -- (202.5,88) -- (81.5,88) -- cycle  ;
\draw (142,76) node    {$\ \ \ \ \ \ \ \ \ IND\ \ \ \ \ \ \ \ \ \ $};
\draw  [color={rgb, 255:red, 11; green, 57; blue, 112 }  ,draw opacity=1 ]  (4,179) -- (208,179) -- (208,203) -- (4,203) -- cycle  ;
\draw (7,183.4) node [anchor=north west][inner sep=0.75pt]    {$\ Weakly\ Isolated\ Singularity$};
\draw (16,97) node [anchor=north west][inner sep=0.75pt]   [align=left] {\cite{Oka2010}};
\draw  [color={rgb, 255:red, 11; green, 57; blue, 112 }  ,draw opacity=1 ]  (116.5,122) -- (161.5,122) -- (161.5,146) -- (116.5,146) -- cycle  ;
\draw (139,134) node    {$PND$};
\draw (143,96) node [anchor=north west][inner sep=0.75pt]   [align=left] {Thm.~\ref{lem:innerpar}};
\draw (140,36) node [anchor=north west][inner sep=0.75pt]   [align=left] {\cite{AraujoBodeSanchez}};
\draw (142,152) node [anchor=north west][inner sep=0.75pt]   [align=left] {Thm.~\ref{strong-isolated3}};
\draw (80,126) node [anchor=north west][inner sep=0.75pt]   [align=left] {\cite{AraujoBodeSanchez}};
\end{tikzpicture}

\label{boundaryfu}
\end{subfigure}
\hfil \hfil
\begin{subfigure}[a]{.4\textwidth}
\tikzset{every picture/.style={line width=0.75pt}} 

\begin{tikzpicture}[x=0.75pt,y=0.75pt,yscale=-1,xscale=1]

\draw [line width=0.75]    (135,36) -- (135,52)(132,36) -- (132,52) ;
\draw [shift={(133.5,60)}, rotate = 270] [color={rgb, 255:red, 0; green, 0; blue, 0 }  ][line width=0.75]    (10.93,-3.29) .. controls (6.95,-1.4) and (3.31,-0.3) .. (0,0) .. controls (3.31,0.3) and (6.95,1.4) .. (10.93,3.29)   ;
\draw [line width=0.75]    (49,33.99) -- (49.94,168.99)(46,34.01) -- (46.94,169.01) ;
\draw [shift={(48.5,177)}, rotate = 269.6] [color={rgb, 255:red, 0; green, 0; blue, 0 }  ][line width=0.75]    (10.93,-3.29) .. controls (6.95,-1.4) and (3.31,-0.3) .. (0,0) .. controls (3.31,0.3) and (6.95,1.4) .. (10.93,3.29)   ;
\draw [line width=0.75]    (103,94.02) -- (102.1,169.02)(100,93.98) -- (99.1,168.98) ;
\draw [shift={(100.5,177)}, rotate = 270.69] [color={rgb, 255:red, 0; green, 0; blue, 0 }  ][line width=0.75]    (10.93,-3.29) .. controls (6.95,-1.4) and (3.31,-0.3) .. (0,0) .. controls (3.31,0.3) and (6.95,1.4) .. (10.93,3.29)   ;
\draw [line width=0.75]    (136,93) -- (136,113)(133,93) -- (133,113) ;
\draw [shift={(134.5,121)}, rotate = 270] [color={rgb, 255:red, 0; green, 0; blue, 0 }  ][line width=0.75]    (10.93,-3.29) .. controls (6.95,-1.4) and (3.31,-0.3) .. (0,0) .. controls (3.31,0.3) and (6.95,1.4) .. (10.93,3.29)   ;
\draw [line width=0.75]    (138,150) -- (138,170)(135,150) -- (135,170) ;
\draw [shift={(136.5,178)}, rotate = 270] [color={rgb, 255:red, 0; green, 0; blue, 0 }  ][line width=0.75]    (10.93,-3.29) .. controls (6.95,-1.4) and (3.31,-0.3) .. (0,0) .. controls (3.31,0.3) and (6.95,1.4) .. (10.93,3.29)   ;

\draw  [color={rgb, 255:red, 11; green, 57; blue, 112 }  ,draw opacity=1 ]  (1,6) -- (209,6) -- (209,30) -- (1,30) -- cycle  ;
\draw (105,18) node    {$\ \ \ \ \ \ \ \ \ Convenient+SND\ \ \ \ \ \ \ \ \ $};
\draw  [color={rgb, 255:red, 11; green, 57; blue, 112 }  ,draw opacity=1 ]  (81.5,64) -- (202.5,64) -- (202.5,88) -- (81.5,88) -- cycle  ;
\draw (142,76) node    {$\ \ \ \ \ \ \ \ \ SIND\ \ \ \ \ \ \ \ \ \ $};
\draw  [color={rgb, 255:red, 11; green, 57; blue, 112 }  ,draw opacity=1 ]  (4,179) -- (208,179) -- (208,203) -- (4,203) -- cycle  ;
\draw (7,183.4) node [anchor=north west][inner sep=0.75pt]    {$\ \ \ \ \ \ \ \ Isolated\ Singularity\ \ \ \ \ \ \ $};
\draw (16,97) node [anchor=north west][inner sep=0.75pt]   [align=left] {\cite{Oka2010}};
\draw  [color={rgb, 255:red, 11; green, 57; blue, 112 }  ,draw opacity=1 ]  (116.5,122) -- (161.5,122) -- (161.5,146) -- (116.5,146) -- cycle  ;
\draw (139,134) node    {$SPND$};
\draw (143,96) node [anchor=north west][inner sep=0.75pt]   [align=left] {Thm.~\ref{lem:innerpar}};
\draw (140,36) node [anchor=north west][inner sep=0.75pt]   [align=left] {\cite{AraujoBodeSanchez}};
\draw (142,152) node [anchor=north west][inner sep=0.75pt]   [align=left] {Thm.~\ref{strong-isolated3}};
\draw (80,126) node [anchor=north west][inner sep=0.75pt]   [align=left] {\cite{AraujoBodeSanchez}};
\end{tikzpicture}

\label{boundaryfu}
\end{subfigure}
\caption{Relations between the different types of non-degeneracy and isolatedness of singularities.}
\label{fig:diagram}
\end{figure}

We provide a number of examples that show that for all implications in the diagram the converse is not true in general. This is in contrast to the complex setting, where (strong) inner non-degeneracy and (strong) partial non-degeneracy are equivalent.
\begin{teo}\label{teo:examples}
None of the implications in Figure~\ref{fig:diagram} is an equivalence.
\end{teo}

Having proved that strong partial non-degeneracy implies the existence of an isolated singularity, we would like to know how far this implication is from being an equivalence. We prove that if we assume that a polynomial map satisfies some extra properties, then the notion of strong partial non-degeneracy is equivalent to the existence of an isolated singularity. We find such conditions for the semiholomorphic case (Proposition~\ref{prop:strongboundary}) and the general mixed case (Proposition~\ref{prop:strongboundary2}). As a consequence we obtain a criterion of non-isolatedness of the singularity at the origin (Corollary~\ref{non-isolated}), which emphasizes the known fact that isolated singularities are very rare. To our knowledge, this is one of the first results that lead to the non-isolatedness of a mixed polynomial from data obtained from a suitable Newton polygon.

All holomorphic polynomials $f:\C^2\to\C$ with an isolated singularity satisfy the strong Milnor condition~(SMC), which means that $f/|f|$ defines a $S^1$-fibration on the complement of $V_f$ in sufficiently small 3-spheres. This property is in general not shared by mixed polynomials with isolated singularities \cite{Milnor1968}. However, we find that strong inner non-degeneracy is a sufficient condition for such a Milnor fibration.

\begin{teo}\label{SMCteo}
Let $f : (\C^2,0) \to (\C,0)$ be a strongly inner non-degenerate mixed polynomial. Then $f$ satisfies the strong Milnor condition.
\end{teo}
This includes the case of convenient and strongly non-degenerate mixed polynomials that was proved in \cite{Oka2010} and the case of radially weighted homogeneous with isolated singularity, which was proved in \cite{AraujoTibar}.

\begin{prop}\label{teo:rad}
Let $f: (\C^2,0) \to (\C,0)$ be a radially weighted homogeneous mixed polynomial. Then the following properties are equivalent:
\begin{itemize}
\item $f$ has a weakly isolated singularity at the origin and satisfies the strong Milnor condition.
\item $f$ has an isolated singularity at the origin.
\item $f$ is strongly partially non-degenerate.
\item $f$ is strongly inner non-degenerate.
\end{itemize}
\end{prop}

The rest of the paper is structured as follows. Section~\ref{sec:prem} reviews the definitions of non-degeneracy of mixed functions introduced in \cite{Oka2010} and \cite{AraujoBodeSanchez}, while Section~\ref{section3} introduces partial non-degeneracy and compares the different notions, proving the implications shown in Figure~\ref{fig:diagram} and stated in Theorem~\ref{lem:innerpar} and Theorem~\ref{strong-isolated3}. We also provide examples that prove Theorem~\ref{teo:examples}. Section~\ref{section5} discusses conditions under which strong partial non-degeneracy becomes equivalent to the existence of an isolated singularity, while Section~\ref{section4} studies the strong Milnor condition, resulting in proofs of Theorem~\ref{SMCteo} and Proposition~\ref{teo:rad}


\vspace{0.3cm}
{\bf Acknowledgments:} B. Bode acknowledges funding from the European Union’s Horizon 2020 Research and Innovation Programme under the Marie Sklodowska-Curie grant agreement No 101023017, and E. Sanchez Quiceno acknowledges the
 supports by grants 2019/11415-3 and 2017/25902-8, São Paulo Research Foundation (FAPESP). This research was started during the visit of the second author to ICMAT, which was financed by FAPESP. We express our gratitude to the institute for their warm reception, particularly to Professor Daniel Peralta Salas for the invaluable support during the visit. The authors are thankful to Professor Osamu Saeki from Kyushu University, Japan, and Professor Raimundo N. Araújo dos Santos from ICMC-USP, Brazil, for their valuable discussions and comments that contributed to the paper.   

\section{PRELIMINARIES}\label{sec:prem}
In this section we review some background on mixed singularities and mixed hypersurfaces. A more detailed account of the concepts and tools concerning mixed polynomials  and their Newton boundaries can be found in \cite{AraujoBodeSanchez, Oka2010}.

We consider the germ of a mixed polynomial $f:(\C^2,0)\to (\C,0)$, $$f(z,\bar{z})=\sum_{\nu,\mu} c_{\nu,\mu}z^{\nu}\bar{z}^{\mu},$$
where $z = (u,v)$, $ \bar{z}= (\bar{u},\bar{v})$, $z^{\nu} = {u}^{\nu_{1}}v^{\nu_{2}}$ for $\nu = (\nu_{1} ,\nu_{2})$ (respectively $\bar{z}^{\mu}=\bar{u}^{\mu_{1}}{\bar{v}}^{\mu_{2}}$ for $\mu=(\mu_{1},\mu_{2})$). 
The support of $f$ is defined as $supp(f) := \{\nu+\mu : c_{\nu,\mu}\neq 0\} \subset \N^2$. 
For mixed polynomials $f$ we consider the singular set, the solutions of  
\begin{equation}\label{criticalpoint}
    \Sigma_f:=
\begin{cases} 
&s_{1,f}:= f_{u} \overline{f_{\bar{v}}}- \overline{f_{\bar{u}}}f_{v}=0, \\
 &s_{2,f}:=|f_{u}|^2-|f_{\bar{u}}|^2=0 \\
 &s_{3,f}:=|f_{v}|^2-|f_{\bar{v}}|^2=0
   \end{cases}
\end{equation}
as a germ of sets at the origin, where $f_x$ denotes the partial derivative with respect to $x$. Note that this definition of the singular set is equivalent to the usual one \cite{AraujoBodeSanchez}, which was also given in the introduction. When $f$ is $u$-semiholomorphic, i.e., $f$ does not depend on the variable $\bar{u}$, the singular set can be defined as the solution of 
\begin{equation}\label{criticalpointsem}
    \Sigma_f:=
\begin{cases} 
 &f_u=0 \\
 &s_{3,f}=0
   \end{cases}
\end{equation}
as a germ of sets at the origin. 
If there is a neighbourhood $U$ of the origin $0\in \C^2$ with $U\cap\Sigma_f=\{0\}$, we say that the origin is an isolated mixed singularity of $f$. If there is a neighbourhood $U$ of the origin $0$ with $U\cap\Sigma_f\cap V_f=\{0\}$, we say that the origin is a weakly isolated mixed singularity of $f$. 
\vspace{0.2cm}

\noindent Oka \cite{Oka2010} defined the Newton boundary $\Gamma_f$ of a mixed polynomial, which in the two variable case is formed by 0-dimensional (“vertices” ) and compact 1-dimensional (“edges”) faces.  
 The \textit{Newton principal part} $f_{\Gamma}$ of $f$ is defined by
\begin{equation}\label{eq:ffboun}
    f_{\Gamma}(z,\bar{z})=\sum_{\nu+\mu \in \Gamma_f} c_{\nu,\mu} z^\nu \bar{z}^\mu.
\end{equation} 
When $f\equiv f_{\Gamma}$, we say that $f$ is a \textit{boundary polynomial}.
\vspace{0.3cm}

For a positive weight vector $P=(p_1,p_2)\in\mathbb{N}^2$ we can define the radial degree of each monomial of $f$ relative to $P$ by setting $rdeg_P (M_{\nu,\mu}):= \sum^2_{j=1} p_j (\nu_j + \mu_j)$, where $M_{\nu,\mu}=c_{\nu,\mu}z^{\nu}\bar{z}^{\mu}$.

To every $P$ we associate a face function $f_P$ which corresponds to the monomials of $f$ on which $rdeg_P$ is minimal among all monomials of $f$. We denote the corresponding minimal value by $d(P;f)$. This face function can be seen as the restriction of $f$ to one face $\Delta(P)$ of $\Gamma_f$
  $$f_{P}(z,\bar{z})=f_{\Delta(P)}(z,\bar{z})\defeq \sum_{\nu+\mu \in \Delta(P)} c_{\nu,\mu}z^{\nu}\bar{z}^{\mu}.$$  
  
We call $f(z, \bar{z})$ a \emph{radially weighted homogeneous polynomial of radial type} if there is a positive weight vector $P\in\mathbb{N}^2$ such that $f=f_P$.
   
The face functions play an important role in the study of the topology and the singularities of $f$ \cite{Oka2010,AraujoBodeSanchez}.
\begin{definition}[Oka \cite{Oka2010}]
A face function $f_{\Delta}$ of a mixed polynomial $f$ for some compact face $\Delta$ (0- or 1-dimensional) of $\Gamma_f$ is called {\bf Newton non-degenerate (ND)} if $V_{f_{\Delta}}\cap\Sigma_{f_{\Delta}}\cap(\mathbb{C}^*)^2=\emptyset$. It is called {\bf strongly Newton non-degenerate (SND)} if $\Sigma_{f_{\Delta}}\cap(\mathbb{C}^*)^2=\emptyset$. We say that $f$ is (strongly) Newton non-degenerate if $f_{\Delta}$ is (strongly) Newton non-degenerate for all compact faces $\Delta$ of $\Gamma_f$.
\end{definition}

Let $f:\C^2\to\C$ be a mixed polynomial with a boundary that has at least one compact 1-face. A 0-dimensional face of the boundary of $\Gamma_+(f)$ that bounds a compact 1-face is called an \textit{extreme vertex} if it is the boundary of a unique 1-face. Otherwise, we call it a \textit{non-extreme vertex}. Let $\{P_1,P_2,\ldots,P_N\}$ be the sequence of weight vectors for which $\Delta(P_i)$ is a compact 1-face of $\Gamma_f$, ordered as in \cite{AraujoBodeSanchez}, i.e., $P_i=(p_{i,1},p_{i,2})\succ P_j=(p_{j,1},p_{j,2})$ if and only if $k_i>k_j$, where $k_\ell=\tfrac{p_{\ell,1}}{p_{\ell,2}}$. With these notions we consider the following definitions.

\begin{definition}[{\cite[Definition 3.1]{AraujoBodeSanchez}}]\label{Newtoncond}
We say that $f$ is an {\bf inner Newton non-degenerate (IND)} if both of the following conditions hold:
\begin{itemize}
    \item[(i)] the face functions $f_{P_1}$ and $f_{P_N}$ have no critical points in $V_{f_{P_1}} \cap (\C^2\setminus \{v=0\})$ and $V_{f_{P_N}}\cap (\C^2\setminus \{u=0\})$, respectively.
\item[(ii)] for each 1-face and non-extreme vertex $\Delta$, the face function $f_\Delta$ has no critical points in $V_{f_\Delta} \cap (\C^*)^{2}$.
\end{itemize}
\end{definition}
As in the other notions of non-degeneracy, there is also a related ``strong'' version of the property.
\begin{definition}[{\cite[Definition 6.1]{AraujoBodeSanchez}}]\label{def-strong}
We say that $f$ is a {\bf strongly inner Newton non-degenerate (SIND)} if both of the following conditions hold:
\begin{itemize}
    \item[(i)] the face functions $f_{P_1}$ and $f_{P_N}$ have no critical points in $ \C^2\setminus \{v=0\}$ and $\C^2\setminus \{u=0\}$, respectively.
\item[(ii)] for each 1-face and non-extreme vertex $\Delta$, the face function $f_\Delta$ has no critical points in $(\C^*)^{2}$.
\end{itemize}
\end{definition} 

We may extend both definitions to mixed polynomials whose Newton boundary consists of a single vertex, so that $N=0$, by saying that such a mixed polynomial $f$ is strongly inner Newton non-degenerate if $f_{\Gamma}$ has no critical points in $\mathbb{C}^2\backslash\{(0,0)\}$. It is inner Newton non-degenerate if $f_{\Gamma}$ has no critical points in $V_{f_{\Delta}}\cap(\mathbb{C}^2\backslash\{(0,0)\})$. Such functions are for example included in the complex setting \cite{mondal}. The proofs of our results are often written in terms of the face functions $f_{P_i}$ associated with 1-faces. The same arguments apply to a mixed polynomial $f$ with a Newton boundary without compact 1-face if we set $P_1=P_N=(1,1)$, $f_{P_1}=f_{P_N}=f_{\Gamma}$, as well as $k_1=k_N=1$. However, the origin is not necessarily a singular point of such polynomials, as is shown by the example $f(u,v)=u$. Hence the corresponding results for these functions should be interpreted as ``If $f$ has a singular point at the origin, it is an isolated singularity.'' as opposed to simply ``$f$ has an isolated singularity at the origin.'', for example.

In \cite{AraujoBodeSanchez} we proved that the concept of an inner non-degenerate mixed polynomial is a natural generalization of the definition of a non-degenerate, convenient mixed polynomial. Furthermore, inner non-degenerate mixed polynomials have weakly isolated singularities. The analogous statement for the family of non-degenerate, convenient mixed polynomials was originally proved by Oka \cite{Oka2010}. Likewise, strongly inner non-degenerate mixed polynomials have isolated singularities and generalize convenient, strongly non-degenerate mixed polynomials. 

\section{(STRONGLY) INNER AND PARTIALLY NON-DEGENERATE MIXED POLYNOMIALS}\label{section3}
In this section we introduce the notion of partial non-degeneracy and strong partial non-degeneracy of mixed functions. We study their relations with inner non-degeneracy and strong inner non-degeneracy of \cite{AraujoBodeSanchez} as well as Oka's non-degeneracy and strong non-degeneracy.

\begin{definition}\label{Newtoncond}
We say that $f$ is {\bf partially Newton non-degenerate (PND)} if both of the following conditions hold for every positive weight vector $P$:
\begin{itemize}
    \item[(i)] the mixed polynomials 
    \begin{equation}\label{PNDiv}
    (s_{1,f}(0,v))_P=(s_{2,f}(0,v))_P=(s_{3,f}(0,v))_P=(f(0,v))_P=0 
    \end{equation}
      have no common solutions in $(\C^*)^2$
     and 
     \begin{equation}\label{PNDiu}
     (s_{1,f}(u,0))_P=(s_{2,f}(u,0))_P=(s_{3,f}(u,0))_P=(f(u,0))_P=0  
     \end{equation} 
  have no common solutions in $(\C^*)^2$.
\item[(ii)] the mixed polynomials 
    \begin{equation}\label{PNDii}
    (s_{1,f})_P=(s_{2,f})_P=(s_{3,f})_P=f_P=0  
    \end{equation}
     have no common solution in $(\C^*)^2$.
\end{itemize}
\end{definition}

Note that the functions in condition (i) only depend on one of the two complex variables, since the other has been set to 0. The condition that Eq.~\eqref{PNDiv} has no solutions in $(\mathbb{C}^*)^2$ is thus equivalent to the condition that these functions (considered as functions in $v$) have no common solutions in $\mathbb{C}^*$. The analogous statement for Eq.~\eqref{PNDiu} holds. 

\begin{definition}\label{Newtoncondstrong}
We say that $f$ is {\bf strongly partially Newton non-degenerate (SPND)} if both of the following conditions hold for every positive weight vector  $P$:
\begin{itemize}
    \item[(i)] the mixed polynomials 
    \begin{equation}\label{SPNDiv}
    (s_{1,f}(0,v))_P=(s_{2,f}(0,v))_P=(s_{3,f}(0,v))_P=0 
    \end{equation}
      have no common solutions in $(\C^*)^2$
     and 
     \begin{equation}\label{SPNDiu}
     (s_{1,f}(u,0))_P=(s_{2,f}(u,0))_P =(s_{3,f}(u,0))_P=0 
     \end{equation} 
  have no common solutions in $(\C^*)^2$
\item[(ii)] the mixed polynomials 
    \begin{equation}\label{SPNDii}
    (s_{1,f})_P=(s_{2,f})_P=(s_{3,f})_P=0
    \end{equation}
     have no common solution in $(\C^*)^2$.
\end{itemize}
\end{definition}
\noindent Both definitions reduce to that of partial non-degeneracy of holomorphic functions if $f$ is holomorphic. Note that the functions $s_{i,f}(0,v)$, $s_{i,f}(u,0)$ with $i=1,2,3$, $f(0,v)$ and $f(u,0)$ are mixed polynomials whose Newton boundary consists of a single vertex. Therefore, the condition (i) in both Definition~\ref{Newtoncond} and Definition~\ref{Newtoncondstrong} is independent of the choice of weight vector $P$.  

We want to compare the new definitions with inner non-degeneracy. As in corresponding definitions for holomorphic polynomials, partial non-degeneracy refers to properties of the face functions of $s_{i,f}$, the functions that define the singular set of $f$, while inner non-degeneracy is concerned with $s_{i,f_P}$, the functions that define the singular set of $f_P$, a face function of $f$. This implies one crucial difference between the two notions. While it is obvious from the definition that (strong) inner non-degeneracy only depends on the Newton principal part of a mixed function, the same is not true for (strong) partial non-degeneracy. This is a manifestation of the fact that taking derivatives and taking $P$-parts are operations that do not commute, so that for example $(f_u)_P\neq (f_P)_u$ in general, as was already explored in \cite{AraujoBodeSanchez}.

In order to prove Theorem~\ref{lem:innerpar}, i.e., that (strong) inner non-degeneracy implies (strong) partial degeneracy, we use the same idea as in the proof of Proposition~3.6 in \cite{AraujoBodeSanchez}. 

\begin{proof}[Proof of Theorem~\ref{lem:innerpar}]
We prove the statement for the strong types of non-degeneracy. The analogous statement for IND and PND follows the same line of reasoning.

We assume that $f$ is SIND, but not SPND, and prove the theorem by contradiction. Suppose that SPND-(ii) does not hold for some weight vector $P=(p_1,p_2)$, with $k_1\geq \frac{p_1}{p_2} \geq k_N$. Then there is a point $(a,b)\in\C^2$ with $ab\neq0$ that is a solution of the system \eqref{SPNDii} 
\begin{equation}\label{eqcsl}
(s_{1,f})_P(a,b)=(s_{2,f})_P(a,b)=(s_{3,f})_P(a,b)=0. 
\end{equation}
This can be written in terms of the face functions of the partial derivatives of $f$:
\begin{equation}\label{eqweak1}
(s_{1,f})_P(a,b)=
\begin{cases}  
((f_{u})_P \cdot (\overline{f_{\bar{v}}})_P)(a,b),&\text{ if } \text{\small $d(P;f_u)+d(P;f_{\bar{v}})<d(P;f_v)+d(P;f_{\bar{u}})$} \\
-((f_{v})_P(\overline{f_{\bar{u}}})_P)(a,b),&\text{ if } \text{\small $d(P;f_v)+d(P;f_{\bar{u}})<d(P;f_u)+d(P;f_{\bar{v}})$} \\
(s_{1,f})_P(a,b),&\text{ if } \text{\small $d(P;f_u)+d(P;f_{\bar{v}})=d(P;f_v)+d(P;f_{\bar{u}})$}
   \end{cases}
\end{equation}
\begin{equation} \label{eqweak2}
(s_{2,f})_P(a,b)=
\begin{cases} 
|(f_{u})_P(a,b)|^2, &\text{ if } \text{\small  $d(P;f_u)<d(P;f_{\bar{u}})$}\\
-|(f_{\bar{u}})_P(a,b)|^2,&\text{ if } \text{\small $d(P;f_{\bar{u}})<d(P;f_u)$} \\
(s_{2,f})_P(a,b),&\text{ if } \text{\small $d(P;f_{u})=d(P;f_{\bar{u}})$} 
   \end{cases}
\end{equation}

\begin{equation} \label{eqweak3}
(s_{3,f})_P(a,b)=
\begin{cases} 
|(f_{v})_P(a,b)|^2, &\text{ if } \text{\small  $d(P;f_v)<d(P;f_{\bar{v}})$}\\
-|(f_{\bar{v}})_P(a,b)|^2, &\text{ if } \text{\small  $d(P;f_{\bar{v}})<d(P;f_v)$} \\
(s_{3,f})_P(a,b), &\text{ if } \text{\small $d(P;f_{v})=d(P;f_{\bar{v}})$} 
     \end{cases}
\end{equation}
where $((f_{u})_P\cdot (\overline{f_{\bar{v}}})_P)(a,b)=(f_{u})_P(a,b)\cdot (\overline{f_{\bar{v}}})_P(a,b)$. Since $k_1\geq \frac{p_1}{p_2} \geq k_N$, the face function $f_P$ is neither of type $u$- and $\bar{u}$-semiholomorphic nor $v$- and $\bar{v}$-semiholomorphic. Therefore, applying \cite[Lemma~3.5]{AraujoBodeSanchez} we have 
\begin{equation*}
\begin{cases} 
&d(P;f_u)> d(P;f_{\bar{u}})=d(P;f)-p_1, \text{ if $f_P$ is $\bar{u}$-semiholomorphic} \\
 &d(P;f_{\bar{u}})> d(P;f_u)=d(P;f)-p_1, \text{ if $f_P$ is $u$-semiholomorphic} \\
 &d(P;f_{\bar{u}})=d(P;f_u)=d(P;f)-p_1, \text{ if $f_P$ depends on both $u$ and $\bar{u}$}
   \end{cases}
\end{equation*}
and 
\begin{equation*}
\begin{cases} 
&d(P;f_v)> d(P;f_{\bar{v}})=d(P;f)-p_2, \text{ if $f_P$ is $\bar{v}$-semiholomorphic} \\
 &d(P;f_{\bar{v}})> d(P;f_v)=d(P;f)-p_2, \text{ if $f_P$ is $v$-semiholomorphic} \\
 &d(P;f_{\bar{v}})=d(P;f_v)=d(P;f)-p_2, \text{ if $f_P$ depend in $v$ and also $\bar{v}$}.
   \end{cases}
\end{equation*}
Depending on the combination of these different cases, there are 9 cases to consider in total.
We show the calculation for the case that $d(P;f_u)> d(P;f_{\bar{u}})=d(P;f)-p_1$ and $d(P;f_{\bar{v}})=d(P;f_v)=d(P;f)-p_2$. The other cases with $k_1\geq \frac{p_1}{p_2} \geq k_N$ follow analogously. Comparing Eqs.~\eqref{eqweak1}-\eqref{eqweak3} with Eq.~\eqref{eqcsl} we find 
\begin{align}\label{redusys}
 &(s_{1,f})_P(a,b)=(f_{v})_P(a,b) \cdot (\overline{f_{\bar{u}}})_P(a,b)=0 \nonumber\\
 &(s_{2,f})_P(a,b)=|(f_{\bar{u}})_P(a,b)|^2=0  \\
 & (s_{3,f})_P(a,b)=0.\nonumber 
\end{align}
Applying \cite[Lemma~3.5]{AraujoBodeSanchez} in these cases we have $(f_{\bar{u}})_P= (f_P)_{\bar{u}}$, $(f_v)_P= (f_P)_v$ and $(f_{\bar{v}})_P= (f_P)_{\bar{v}}$. Moreover, $f_P$ is a $\bar{u}$-semiholomorphic polynomial, i.e., $(f_P)_{u}\equiv 0$. Thus, we have two subcases to consider: 

If $s_{3,f_P}\not\equiv 0$, then 
\begin{equation}\label{s3fP}
s_{3,f_P}= |(f_P)_v|^2-|(f_P)_{\bar{v}}|^2= |(f_v)_P|^2-|(f_{\bar{v}})_{P}|^2\not\equiv 0.
\end{equation}
 Note that, 
\begin{align}\label{faces3f}
s_{3,f}=&((f_v)_P+M_1)\overline{((f_v)_P+M_1)}-((f_{\bar{v}})_P+M_2)\overline{((f_{\bar{v}})_P+M_2)} \nonumber\\=&
(f_v)_P (\overline{f_v})_P
+ (f_v)_P \overline{M_1}  +  M_1 (\overline{f_{v}})_P +M_1 \overline{M_1}-\nonumber\\ & ((f_{\bar{v}})_P (\overline{f_{\bar{v}}})_P  + (f_{\bar{v}})_P  \overline{M_2}+ M_2(\overline{f_{\bar{v}}})_P+  M_2 \overline{M_2}),
\end{align}
where $M_1=f_v-(f_v)_P$ and $M_2=f_{\bar{v}}-(f_{\bar{v}})_P$ are mixed polynomials satisfying 
$d(P;(f_v)_P)<d(P; M_1)$ and $d(P;(f_{\bar{v}})_P)<d(P; M_2)$ if $M_1\not\equiv0$ and $M_2\not\equiv0$, respectively. Thus, by Eq.~\eqref{s3fP} we have that $(s_{3,f})_P=(f_v)_P (\overline{f_v})_P -(f_{\bar{v}})_P (\overline{f_{\bar{v}}})_P=s_{3,f_P}$. Therefore, it follows that 
\begin{align}
s_{1,f_P}&=(f_P)_u \overline{(f_P)_{\bar{v}}}-  \overline{(f_P)_{\bar{u}}} (f_P)_v =-\overline{(f_P)_{\bar{u}}}(f_P)_v=-\overline{(f_{\bar{u}})_P}(f_v)_P=(s_{1,f})_P\label{equalitiesface1}  \\
 s_{2,f_P}&= |(f_P)_u|^2-|(f_P)_{\bar{u}}|^2= -|(f_P)_{\bar{u}}|^2=-|(f_{\bar{u}})_P|^2=(s_{2,f})_P \label{equalitiesface2} \\
s_{3,f_P}&=(s_{3,f})_P,\label{equalitiesface3} 
\end{align}
and a solution $(a,b)$ in Eq.~\eqref{redusys} implies a solution $s_{1,f_P}(a,b)=s_{2,f_P}(a,b)=s_{3,f_P}(a,b)=0$, which contradicts  SIND-(ii).  

If $s_{3,f_P}\equiv 0$, then the system $s_{1,f_P}=s_{2,f_P}=s_{3,f_P}=0$ is equal to $s_{1,f_P}=s_{2,f_P}=0$. Since Eqs.~\eqref{equalitiesface1}-\eqref{equalitiesface2} still hold in this case, the same solution $(a,b)$ in Eq.~\eqref{redusys}, $(s_{1,f})_P(a,b)=(s_{2,f})_P(a,b)=(s_{3,f})_P(a,b)=0$, implies a solution $s_{1,f_P}(a,b)=s_{2,f_P}(a,b)=0$, which contradicts SIND-(ii).  Therefore, the weight vector $P$ satisfies $\frac{p_1}{p_2}>k_1$,  $\frac{p_1}{p_2}<k_N$ with $ab\neq 0$ or SPND-(i) does not hold.   
\vspace{0.3cm}

In order to deal with these remaining cases we may write $f_{P_1}(u,\bar{u},v,\bar{v})=g(u,\bar{u},v,\bar{v})+B(v,\bar{v})u+C(v,\bar{v})\bar{u}+A(v,\bar{v})$, where $A(v,\bar{v})$ consists of all those terms of $f_{P_1}$ that depend neither on $u$ nor on $\bar{u}$, $B(v,\bar{v})u$ is its linear term in $u$, $C(v,\bar{v})\bar{u}$ its linear term in $\bar{u}$ and $g(u,\bar{u},v,\bar{v})$ is the sum of all remaining terms in $f_{P_1}$. The condition that $f_{P_1}$ has no critical point in $\C^2 \setminus \{v=0\}$ is equivalent to the following system having no solutions:
\vspace{0.3cm}
If $f$ is $v$-convenient,
\begin{equation}\label{almostnondeg1}
\begin{cases} 
&B(v,\bar{v}) \overline{A_{\bar{v}}(v,\bar{v})}-\overline{C(v,\bar{v})}A_v(v,\bar{v})  =0 \\
 &|B(v,\bar{v})|^2-|C(v,\bar{v})|^2=0 \\
 &|A_v(v,\bar{v}) |^2-|A_{\bar{v}}(v,\bar{v})|^2=0.\\ 
   \end{cases}
\end{equation}
If $f$ is not $v$-convenient, i.e., $A(v,\bar{v})\equiv 0$, 
\begin{equation}\label{almostnondeg2}
\begin{cases} 
 &|B(v,\bar{v})|^2-|C(v,\bar{v})|^2=0. 
   \end{cases}
\end{equation}

Now, suppose that SPND-(i) or SPND-(ii) is not satisfied for a weight vector $P=(p_1,p_2)$ with $\frac{p_1}{p_2}>k_1$ and the system \eqref{SPNDii} has a solution of the form $(a,b)$ or $(0,b)$. Then we get:

\vspace{0.3cm}

If $f$ is $v$-convenient,  
\begin{equation}\label{conv}
\begin{cases} 
&B(b,\bar{b}) \overline{A_{\bar{v}}(b,\bar{b})}-\overline{C(b,\bar{b})}A_v(b,\bar{b})  =0 \\
 &|B(b,\bar{b})|^2-|C(b,\bar{b})|^2=0 \\
 &|A_v(b,\bar{b}) |^2-|A_{\bar{v}}(b,\bar{b})|^2=0.\\ 
   \end{cases}
\end{equation}
\vspace{0.3cm}

If $f$ is not $v$-convenient, i.e., $A(v,\bar{v})\equiv 0$, 
\begin{equation}\label{noconv}
\begin{cases} 
&B(b,\bar{b}) (\overline{B_{\bar{v}}(b,\bar{b})a+C_{\bar{v}}(b,\bar{b})\bar{a} })-(B_{v}(b,\bar{b})a+C_{v}(b,\bar{b})\bar{a})\overline{C(b,\bar{b})}  =0\\
 &|B(b,\bar{b})|^2-|C(b,\bar{b})|^2=0 \\
 &|B_{v}(b,\bar{b})a+C_{v}(b,\bar{b})\bar{a}|^2 -|B_{\bar{v}}(b,\bar{b})a+C_{\bar{v}}(b,\bar{b})\bar{a}|^2=0.\\
   \end{cases}
\end{equation}
In both cases, convenient and non-convenient,  a solution of Eq.~\eqref{conv} or Eq.~\eqref{noconv} with $(a,b), b\neq0$ gives a solution of Eq.~\eqref{almostnondeg1} and Eq.~\eqref{almostnondeg2}, respectively, which is a contradiction to SIND-(i). Analogously,  we can use SIND-(i) in $f_{P_N}$ to see the case $\frac{p_1}{p_2}<k_N$ and the case where Eq.~\eqref{SPNDiu} has a solution of the form $(a,0)$. In either case we obtain a contradiction. Recall that the condition SPND-(i) is independent of the weight vector $P$, so if SPND-(i) fails, it fails for a weight vector of the form discussed above. Therefore, our assumption that $f$ is not SPND was wrong, which proves that SIND implies SPND.

The proof of the implication from IND to PND follows the same argument. The only difference is that there is one more equation $f_P=0$ in the systems of equations above.
\end{proof}
The following example shows that partially non-degenerate mixed functions are not necessarily inner non-degenerate.

\begin{ex}\label{ex1}
The mixed polynomial $f(u,\bar{u},v,\bar{v})=v\bar{v}-u\bar{u}+\bar{v}u^2$ is partially non-degenerate, but not inner non-degenerate.

Its Newton boundary only has one compact 1-face with weight vector $P=(1,1)$ and $f_P(u,\bar{u},v,\bar{v})=v\bar{v}-u\bar{u}$. This shows that $f$ is not inner non-degenerate (or ``inner degenerate''), since every point in $\mathbb{C}^2$ is a singular point of $f_P$.

However, $(f(0,v))_P=f(0,v)=v\bar{v}$ and $(f(u,0))_P=f(u,0)=u\bar{u}$, which both have no zeros in $\mathbb{C}^*$, so that condition (i) in Definition~\ref{Newtoncond} is satisfied.

If $(u_*,v_*)\in(\mathbb{C}^*)^2$ is a root of $f_P$ for any weight vector $P$, then $P=(1,1)$ and $|u_*|=|v_*|$. We calculate for $P=(1,1)$
\begin{equation}
(s_{3,f})_P(u_*,\overline{u_*},v_*,\overline{v_*})=-\overline{u_*}^2v_*-u_*\overline{v_*}^2.
\end{equation}
If this equals 0, we have $\text{Re}(u_*^2\overline{v_*})=0$ and this implies
\begin{equation}
2\varphi_1-\varphi_2=\pm \pi \text{ mod }2\pi,
\end{equation}
where $\varphi_1=\arg(u_*)$ and $\varphi_2=\arg(v_*)$. Substituting this into $(s_{1,f})_P$ leads to
\begin{align}
(s_{1,f})_P(u_*,v_*)&=-u_*^3+2u_*\overline{v_*}^2\nonumber\\
&=|u_*|^3 \rme^{-3\varphi_1\rmi}=\overline{u_*}^3,
\end{align}
since $|u_*|=|v_*|$.
Since $u_*\neq 0$, there is no common solution to Eq.~\eqref{PNDii} in $(\mathbb{C}^*)^2$ and hence $f$ is partially non-degenerate. 
\end{ex}

In fact, examples of polynomials that are partially non-degenerate, but not inner non-degenerate, can also be found among the family of semiholomorphic polynomials, such as $f(u,v,\bar{v})=u^3-2(v+\bar{v})+\left(\rmi v\bar{v}+\tfrac{1}{4}(v^2-\bar{v}^2)\right)u$. To test the non-degeneracies PND and SPND for the semiholomorphic case we can use the following result. 
\begin{lemma}\label{lem:uP}
Let $f$ be a mixed polynomial  and $P$ a positive weight vector. Then 
\begin{itemize}
\item[(i)] if $f_P$ is $u$-semiholomorphic and $$d(P;(f_u)_P)+ d(P;(f_{\bar{v}})_P)\neq d(P;(f_v)_P) +d(P;(f_{\bar{u}})_P),$$  then 
$(s_{1,f})_P=(s_{2,f})_P=(s_{3,f})_P=0$ if and only if $(f_P)_u=(s_{3,f})_P=0$.
\item[(ii)] if $f$ is $u$-semiholomorphic, then $(s_{1,f})_P=(s_{2,f})_P=(s_{3,f})_P=0$ if and only if $(f_u)_P=(s_{3,f})_P=0.$ 
\end{itemize}
\end{lemma} 
\begin{proof}
\begin{itemize}
\item[(i)] Since $f_P$ is $u$-semiholomorphic, then by  \cite[Lemma~3.5]{AraujoBodeSanchez} we have that 
$d(P;f)-p_1=d(P;(f_u)_P)< d(P;(f_{\bar{u}})_P)$, $(f_P)_u=(f_u)_P$ and $(f_P)_{\bar{u}}\equiv 0$. Thus 
\begin{equation}\label{simplifyeq1}
(s_{2,f})_P= |(f_u)_P|^2=|(f_P)_u|^2=s_{2,f_P}. 
\end{equation}
Since $d(P;(f_u)_P) +d(P;(f_{\bar{v}})_P)\neq d(P;(f_v)_P) +d(P;(f_{\bar{u}})_P)$, we have two cases to consider:  

If $d(P;(f_u)_P)+ d(P;(f_{\bar{v}})_P)< d(P;(f_v)_P) + d(P;(f_{\bar{u}})_P)$, then $(s_{1,f})_P= (f_u)_P (f_{\bar{v}})_P=(f_P)_u (f_{\bar{v}})_P$ and thus $(s_{1,f})_P=(s_{2,f})_P= 0$ if and only if $(f_P)_u=0$. 

If  $d(P;(f_u)_P)+ d(P;(f_{\bar{v}})_P)> d(P;(f_v)_P) + d(P;(f_{\bar{u}})_P)$, then
\begin{equation}\label{simplifyeq2}
(s_{1,f})_P= (f_v)_P (f_{\bar{u}})_P
\end{equation}
and since $d(P;(f_u)_P)< d(P;(f_{\bar{u}})_P)$, we have
$$d(P;(f_v)_P)<d(P;(f_{\bar{v}})_P).$$
This inequality implies that $(s_{3,f})_P= |(f_v)_P|^2$. It follows from Eqs.~\eqref{simplifyeq1}-\eqref{simplifyeq2} that $(s_{1,f})_P=(s_{2,f})_P=(s_{3,f})_P=0$ if and only if $(f_P)_u=(f_v)_P=0$.
\item[(ii)] Since $f_{\bar{u}}\equiv0$, we have 
\begin{align}\label{semicase1}
s_{1,f}&= f_u \overline{f_{\bar{v}} }- \overline{f_{\bar{u}} }f_v= f_u \overline{f_{\bar{v}} } \\\label{semicase2}
 s_{2,f}&= |f_u|^2-|f_{\bar{u}}|^2= |f_u|^2.
\end{align}
Taking the $P$-part of the mixed polynomials of Eqs.~\eqref{semicase1}-\eqref{semicase2} yields
\begin{align}
(s_{1,f})_P&= (f_u)_P(\overline{f_{\bar{v}}})_P\\
 (s_{2,f})_P&= |(f_u)_P|^2.
\end{align}
Therefore, $(f_u)_P=0$ if and only if $(s_{1,f})_P=(s_{2,f})_P=0$.  
\end{itemize}
\end{proof}

\begin{ex}\label{ex2}
The mixed polynomial $f(u,\bar{u},v,\bar{v})=u^{10} + u^2 v + (v \bar{v})^n + \bar{u} v^{2n - 1}$ with $n>1$ is strongly partially non-degenerate, but not strongly inner non-degenerate.

The polynomial has two 1-faces with weight vectors $P_1=(2n-1,2)$ and $P_2=(1,8)$, and
\begin{align*}
f_{P_1}(u,v,\bar{v})&=u^2 v+(v\bar{v})^n\\
f_{P_2}(u,v,\bar{v})&=u^{10} + u^2 v.
\end{align*}
The face function $f_{P_1}$ is semiholomorphic and satisfies the inequality from Lemma~\ref{lem:uP}(i). Therefore, we can find its critical points by solving $(f_{P_1})_u=s_{3,f_{P_1}}=0$. We find that
\begin{align*}
(f_{P_1})_u(u,v,\bar{v})&=2u v\\
s_{3,f_{P_1}}(u,v,\bar{v})&=n\bar{u}^2v^{n-1}\bar{v}^n+(u\bar{u})^2(1+nv^n\bar{v}^{n-1}).
\end{align*}
It follows that $(0,v)$ is a critical point of $f_{P_1}$ for all $v\in\mathbb{C}$. Therefore, $f$ is not strongly inner non-degenerate, since it does not satisfy Condition (i) in Definition~\ref{def-strong}.

To prove SPND for $f$ we first calculate
\begin{align*}
s_{2,f}(u,\bar{u},v,\bar{v})=&100 u^9 \bar{u}^9 + 20 u \bar{u}^9 v + 20 u^9 \bar{u} \bar{v} + 4 u \bar{u} v \bar{v}\\
& -  (v\bar{v})^{2 n-1}
\end{align*}
We find that $s_{2,f}(0,v)=- (v \bar{v})^{2 n-1}$ and $s_{2,f}(u,0)=100 u^9 \bar{u}^9$. Both of these functions have no zeros in $(\mathbb{C}^*)^2$, so that Condition (i) in Definition~\ref{Newtoncondstrong} is satisfied.

Furthermore, we have
\begin{align*}
s_{1,f}(u,\bar{u},v,\bar{v})=& 10 n u^9 v^{n-1} \bar{v}^n-u^2 \bar{v}^{2 n-1} + \bar{u} v^{2 n-2} \bar{v}^{2 n-1}\\
 &- 2 n \bar{u} v^{2 n-2} \bar{v}^{2 n-1} + 2 n u (v \bar{v})^n  - n \bar{v}^{2 n-1} v^{n-1} \bar{v}^n.
\end{align*}

The Newton polygon of this function has two compact 1-faces with $Q_1=(2n-2,1)$, $Q_2=(1,1)$ and
\begin{align}
(s_{1,f})_{Q_1}(u,\bar{u},v,\bar{v})&= 2 n u (v \bar{v})^n- n \bar{v}^{3 n-1} v^{n-1}\label{s1fq1}\\
(s_{1,f})_{Q_2}(u,\bar{u},v,\bar{v})&=-u^2 \bar{v}^{2 n-1}+ 2 n u (v \bar{v})^n.\nonumber
\end{align}
Note that if $P\notin\{Q_1,Q_2\}$, then $(s_{1,f})_P$ consists of exactly one summand of the expressions above. Thus these functions have no zeros in $(\mathbb{C}^*)^2$ and it follows that Condition (ii) of Definition~\ref{Newtoncondstrong} is satisfied for all such $P$. 

For $P=Q_2$ check that $(s_{2,f})_P(u,\bar{u},v,\bar{v})=4 u \bar{u} v\bar{v}$, which is non-zero on $(\mathbb{C}^*)^2$.

For $P=Q_1$ we get
\begin{equation}
(s_{2,f})_P(u,\bar{u},v,\bar{v})=4 u\bar{u} v\bar{v}- (v\bar{v})^{2 n-1}.
\end{equation}
Assume that $(u_*,v_*)\in(\mathbb{C}^*)^2$ is a zero of $(s_{2,f})_P$. Then $|u_*|=\tfrac{1}{4}|v_*|^{2n-2}$. But then the first summand in Eq.~\eqref{s1fq1} has modulus $2n|v_*|^{4n-2}$, while the second summand has modulus $n|v_*|^{4n-2}$. Therefore, $(s_{1,f})_{Q_1}(u_*,v_*)\neq 0$, since $v_*\neq 0$.  

We thus have a contradiction. There is no common zero $(u_*,v_*)\in(\mathbb{C}^*)^2$ of $(s_{1,f})_P$ and $(s_{2,f})_P$ for any weight vector $P$. Therefore, Condition (ii) in Definition~\ref{Newtoncondstrong} is satisfied and $f$ is strongly partially non-degenerate.
\end{ex}

We now prove that PND and SPND imply the existence of a weakly isolated singularity and an isolated singularity, respectively. This follows the same reasoning as the proof of Propositions 3.6 and 6.2 in \cite{AraujoBodeSanchez}.

\begin{proof}[Proof of Theorem~\ref{strong-isolated3}]
 Assume that the origin is not an isolated singularity. Then via the curve selection lemma there exists a real analytic curve $z(\tau)$ of critical points starting at the origin. i.e., a curve $z(\tau)=(u(\tau),v(\tau))=(a\tau^{p_1}+h.o.t., b\tau^{p_2}+h.o.t.), 0 \leq \tau \leq 1$, where $h.o.t.$ refers to higher order terms in $\tau$, satisfying 
\begin{itemize}
    \item[(i)] $z(0)=0$ and $z(\tau)\in \C^2$ for $\tau>0$. 
    \item[(ii)] $s_{1,f}(z(\tau))=s_{2,f}(z(\tau))=s_{3,f}(z(\tau))=0.$  
\end{itemize}
If $u(\tau)\not\equiv 0$ and $v(\tau)\not\equiv 0$, we set $P=(p_1,p_2)$. If $u(\tau)\equiv0$, we set $a=0$ and $P=(0,p_2)$, similarly, $b=0$ and $P=(p_1,0)$ if $v(\tau)\equiv 0$.

Expanding the equations of (ii), we find that the coefficient of $\tau^{d(P;s_{j,f})}$ is $(s_{j,f})_P(a,b)$ for all $j\in\{1,2,3\}$ (compare also the proof of Lemma~\ref{lem:lowestorder}). Since the right hand side of (ii) vanishes, these coefficients have to be zero as well. Thus we get a common solution $(a,b)\in (\C^{*})^2$, $(a,0)$ or $(0,b)$ of the systems \eqref{SPNDii}, \eqref{SPNDiu} or \eqref{SPNDiv}, respectively, which yields a contradiction of SPND-(i) or SPND-(ii). The case of partial non-degeneracy follows in the same way by including the equation $f(z(\tau))=0$. 
  \end{proof}
\begin{corolario}\label{radialequivweak} Let  $f$ be a radially weighted homogeneous polynomial. Then the following conditions are equivalent:
\begin{itemize}
\item[(i)] $f$ has a weakly isolated singularity (isolated singularity) at the origin.
\item[(ii)] $f$ is inner non-degenerate (strongly inner non-degenerate). 
\item[(iii)] $f$ is partially non-degenerate (strongly partially non-degenerate).
\end{itemize}
\end{corolario} 
\begin{proof}
(iii) implies (i) by Proposition~\ref{strong-isolated3}. 
The fact that (ii) implies (iii) follows from Proposition~\ref{lem:innerpar}.
Finally, (i) implies (ii) because if $f$ is radially weighted homogeneous, then the existence of a weakly isolated singularity (isolated singularity) implies conditions IND-(i) and IND-(ii) (SIND-(i) and SIND-(ii)) for the weight vector $P$ associated to the unique 1-face of $\Gamma_f$. 
    \end{proof}

\begin{proof}[Proof of Theorem~\ref{teo:examples}]
In Example 3.4 in \cite{AraujoBodeSanchez} we give an example of a semiholomorphic polynomial that is inner non-degenerate, but Newton degenerate. There are plenty of (strongly) inner non-degenerate mixed polynomials that are not convenient, even among the holomorphic ones (for which (S)IND and (S)PND are equivalent). Consider for example $f(u,v)=u(u^2-v^2)$. It is not convenient, but since $f_u=3u^2-v^2$, $f_v=-2uv$, it is strongly inner non-degenerate.

Example~\ref{ex1} and Example~\ref{ex2} illustrate that in general (S)PND does not imply (S)IND for mixed polynomials.

In \cite{Bode2022semih,bode:thomo} the first author constructs families of semiholomorphic polynomials $f$ with isolated singularities or weakly isolated singularities. In both cases, their Newton boundary consists of a single compact 1-face. There is only one term above the Newton boundary, which only depends on $v$ and $\bar{v}$. The face function $f_P$ corresponding to the 1-face is by construction degenerate, i.e., there exists a $(u_*,v_*)\in(\mathbb{C}^*)^2$ with 
\begin{equation}
f_P(u_*,v_*)=(f_P)_u(u_*,v_*)=s_{3,f_P}(u_*,v_*)=0.
\end{equation} 
Since $f_P$ is semiholomorphic, we have $(f_u)_P=(f_P)_u$ and by construction we have $s_{3,f_P}=(s_{3,f})_P$. It thus follows that $f$ is not SPND. In fact, it is not even PND. A different example (with an isolated singularity whose link is the figure-eight knot) is given in \cite{Rudolph1987}. Since isolated singularities are also weakly isolated, this concludes the proof of Theorem~\ref{teo:examples}. 
\end{proof}

\section{WHEN DOES AN ISOLATED SINGULARITY IMPLY NON-DEGENERACY?}\label{section5}

We have seen in Section~\ref{section3} that SPND mixed polynomials have isolated singularities. The analogous result for SIND mixed polynomials was shown in \cite{AraujoBodeSanchez}. However, we know from Theorem~\ref{teo:examples} that not every isolated singularity comes from such a non-degenerate polynomial. The examples by Rudolph \cite{Rudolph1987} and the first author \cite{bode:thomo} that illustrate this are all semiholomorphic and are not even inner non-degenerate. Therefore, they 
might give the impression that once inner (or partial) non-degeneracy is assumed, strong inner (or partial) non-degeneracy and the existence of an isolated singularity are equivalent. This is not the case. In this section we give an example of a semiholomorphic polynomial with an isolated singularity that is inner non-degenerate and partially non-degenerate, but not strongly inner non-degenerate and not strongly partially non-degenerate. This leads us to a set of conditions on the polynomials that (if satisfied) imply the equivalence between SPND and the existence of an isolated singularity. The example also highlights the need for a distinction between IND and SIND, as well as between PND and SPND, even for semiholomorphic polynomials. Recall that for holomorphic functions in these dimensions all of these notions are equivalent.


First, we recall some results from \cite{AraujoBodeSanchez}.

The face functions play an important role in the study of the topology and the singularities of $f$ \cite{Oka2010,AraujoBodeSanchez} and some of these results come from the following properties of these face functions and their relations with $f$:

Denote the weight vectors whose associated faces are 1-faces of $\Gamma_f$ by $P_i$, $i=1,2,\ldots,N$.
Writing $v=r\rme^{\rmi t}$, $\bar{v}=r\rme^{-\rmi t},$ we may associate to the 1-face $\Delta(P_i)$ corresponding to a weight vector $P_i$ a function that does not depend on $r$ anymore, as follows. For a positive weight vector $P_i=(p_{i,1},p_{i,2})$ let $g_i:\mathbb{C}\times S^1\to\mathbb{C}$ be given by
\begin{equation}\label{resc}
g_i(u,\bar{u},\rme^{\rmi t}):=r^{-\tfrac{d(P_i;f)}{p_{i,2}}}f_{P_i}\left(r^{k_i}u,r^{k_i}\bar{u},r\rme^{\rmi t},r\rme^{-\rmi t}\right),
\end{equation}
where $k_i:=\tfrac{p_{i,1}}{p_{i,2}}$ and $d(P_i;f)$ is the radial degree of $f$ related with $P_i$. Note that the right hand side does not depend on $r$, because $f_{P_i}$ is radially weighted homogeneous with weight vector $P_i$.
 By \cite[Lemma 4.2]{AraujoBodeSanchez} we have  that the function $f$ admits the decomposition
\begin{equation}\label{decomp}
f(u,\bar{u},r\rme^{\rmi t},r\rme^{-\rmi t})=r^{\tfrac{d(P_i;f)}{p_{i,2}}}f_i\left(\frac{u}{r^{k_i}},\frac{\bar{u}}{r^{k_i}},r, t\right),
\end{equation}
where $f_{i}$ is a $r$-parameter deformation of $g_{i},$ that is, the difference $f_i(u,\bar{u},r,t)-g_i(u,\bar{u},\rme^{\rmi t})$ goes to $0,$ as $r\to 0$. If $f$ is radially weighted homogeneous with weight vector $P=P_1$, then $f_1$ does not depend on $r$ and $f_1=g_1$.

If $f$ is $u$-semiholomorphic, we may interpret $g_i$ as a loop in the space of complex polynomials in one variable $u$, whose coefficients are finite Fourier series in $\rme^{\rmi t}$ and $\rme^{-\rmi t}$. We can thus associate to each face function of $f$ a loop in the space of polynomials in one variable. If it is clear which $g_i$ we are referring to, we often drop the index $i$ and simply write $g:\mathbb{C}\times S^1\to\mathbb{C}$ or $g_t:\mathbb{C}\to\mathbb{C}$ with $g_t(u)=g(u,\rme^{\rmi t})$.

Conversely, we may (as in \cite{Bode2019} for example) use loops $g_t$ in the space of polynomials of degree $s$ as above, whose coefficients satisfy certain extra conditions, to obtain radially weighted homogeneous semiholomorphic polynomials. For example, if the loop is $2$-periodic, i.e., all of its coefficients are polynomials in $\rme^{2\rmi t}$ and $\rme^{-2\rmi t}$, so that $g_{t+\pi}=g_t$ for all $t\in[0,2\pi]$, then we can define
\begin{equation}\label{eq:scaling}
f(u,v,\bar{v})=f(u,r\rme^{\rmi t},\rme^{-\rmi t})=r^{sk}g\left(\frac{u}{r^k},\rme^{\rmi t}\right),
\end{equation}
which is a semiholomorphic polynomial for sufficiently large even values of $k\in2\mathbb{Z}$. The constructed semiholomorphic polynomial $f$ has a weakly isolated singularity if and only if the roots of $g_t$ are distinct for all $t\in[0,2\pi]$. It has an isolated singularity if and only if the roots of $g_t$ are distinct for all $t\in[0,2\pi]$ and $\arg(g):(\mathbb{C}\times S^1)\backslash g^{-1}(0)\to S^1$ has no critical points and thus defines a fibration. \cite{Bode2019}.

Since each $g_t$ is holomorphic with respect to $u$, this condition can be expressed in terms of the critical values of the complex polynomials $g_t$. Let $c_j(t)$, $j=1,2,\ldots,s-1$, denote the critical points of $g_t$ and let $v_j(t)=g_t(c_j(t))$ be the corresponding critical values. Then $(u_*,t_*)$ is a critical point of $\arg(g)$ if and only if there is a $j\in\{1,2,\ldots,s-1\}$ such that $u_*=v_j(t_*)$ and $\tfrac{\partial \arg(v_j(t))}{\partial t}|_{t=t_*}=0$.

To every loop $g_t$ in the space of polynomials, we can associate a corresponding loop $\{v_1(t),v_{2},\ldots,v_{s-1}(t)\}$ in the space of critical values $\mathbb{C}^{s-1}/S_{s-1}$, where the symmetric group $S_{s-1}$ acts on an $s-1$-tuple by permutation. A useful technique in the construction of semiholomorphic polynomials is to start with a loop in $V_s:=(\mathbb{C}^*)^{s-1}/S_{s-1}$ that satisfies $\tfrac{\partial \arg(v_j(t))}{\partial t}\neq 0$ for all $t$ and all $j$. Then try to find a corresponding lifted loop in the space of polynomials with the given $v_j(t)$'s as its critical values. If such a lifted loop $g_t$ exists, then its roots are simple for each $t\in[0,2\pi]$ and they thus form a braid $B$. If $g_t$ is also 2-periodic, we can define a semiholomorphic, radially weighted homogeneous polynomial $f$ as in Eq.~\eqref{eq:scaling}. Since $f$ is radially weighted homogeneous, it has an isolated singularity if and only if it strongly inner non-degenerate and strongly partially non-degenerate, which is equivalent to $\arg(g)$ being a fibration, i.e., $\tfrac{\partial \arg(v_j(t))}{\partial t}\neq 0$ for all $t$ and $j$. In this case, the link of the singularity is the closure of $B$.

In the following example, we thus want to start with a loop in $V_s$ that does not have this property, so that the resulting semiholomorphic polynomial $f$ is not SPND and not IPND. However, we want the critical points of $\arg(g)$ to be degenerate, i.e., if $\tfrac{\partial \arg(v_j(t))}{\partial t}|_{t=t_*}= 0$, then $\tfrac{\partial^2 \arg(v_j)}{\partial t^2}(t_*)=0$. We then deform $f$ by adding terms above its Newton boundary that do not change the fact that $f$ is neither SIND nor SPND, but yield an isolated singularity.

\begin{ex}
We start with a loop $v(t)$ in $\mathbb{C}^*$, which is the space of critical values $V_2$ of complex polynomials with distinct roots and degree 2. Following the observation above we want this loop to have argument-critical points, i.e., values $t_*\in[0,2\pi]$ such that $\tfrac{\partial \arg(v)}{\partial t}(t_*)=0$. However, we want these critical points to be degenerate, i.e., $\tfrac{\partial^2 \arg(v)}{\partial t^2}(t_*)=0$.

Take for example $v(t)=\cos(t) - \cos(2 t) + 1-\rmi (\sin(t) - \tfrac{1}{3} \sin(3 t))$, which is shown in Figure~\ref{fig:vt}a) and has argument-critical points at $t=0$ and $t=\pi$. The graph of $\tfrac{\partial\arg(v)}{\partial t}$ is shown in Figure~\ref{fig:vt}b). We may now consider the loop of polynomials $g_t(u)=u^2+v(2t)$. It's critical values are exactly $v(2t)$, i.e., the loop $v(t)$ is traversed twice as $t$ goes from 0 to $2\pi$. Arguments as in \cite[Section 5.1]{Bode2020} show that the roots of this loop of polynomials form a braid that closes to the Hopf link.

\begin{figure}[H]
\centering
\includegraphics[height=4cm]{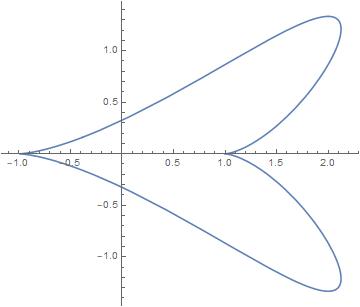}
\includegraphics[height=4cm]{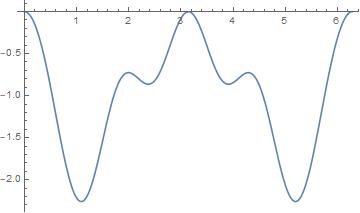}
\caption{a) The parametrized loop $v(t)$. b) The graph of $\tfrac{\partial\arg(v)}{\partial t}$.\label{fig:vt}}
\end{figure}

Since $g_t$ has the required even symmetry, we obtain a corresponding inner non-degenerate, radially weighted homogeneous semiholomorphic polynomial
\begin{align*}
f(u,v,\overline{v})=&u^2+(v\overline{v})^3 + (v\overline{v})^2 \left(\frac{1}{2} (v^2 + \overline{v}^2)\right) - 
 v \overline{v} \left(\frac{1}{2} (v^4 + \overline{v}^4)\right)\nonumber\\
&- \rmi \left((v\overline{v})^2 \left(\frac{1}{2\rmi} (v^2 - \overline{v}^2)\right)- \frac{1}{3} \left(\frac{1}{2\rmi} (v^6 - \overline{v}^6)\right)\right).
\end{align*}
Because $v(2t)$ has argument-critical points, $f$ does not have an isolated singularity and is not strongly inner non-degenerate.

Consider now
\begin{equation*}
F(u,v,\overline{v})=f(u,v,\overline{v}) - \frac{1}{2\rmi}(v-\overline{v})^8.
\end{equation*}
The added term lies above the Newton boundary of $f$, so that $f$ and $F$ have the same principal part and in particular, $F$ is inner non-degenerate, but not strongly inner non-degenerate. Likewise, we find that it is not strongly partially non-degenerate, since
\begin{equation}
(f_u)_P(0,x)=(f_P)_u(0,x)=s_{3,f_P}(0,x)=(s_{3,f})(0,x)=0
\end{equation}
for any $x\in\mathbb{R}$ and $P=(3,1)$.

Note that any singular point $(u,v)$ of $F$ must satisfy $s_{1,F}(u,v,\overline{v})=2u=0$. We can thus find all its singular points by solving $s_{3,F}(0,v,\overline{v})=0$.

Using Mathematica we find that for $-0.01<\re(v)<0.01$ the zeros of $s_{3,F}(0,v)$ consist of the origin and two connected components, whose $\im(v)$-coordinates are parametrized by $\re(v)$ and contain $(\re(v),\im(v))=(0,-0.866025)$ and $(\re(v),\im(v))=(0,0.866025)$, respectively. Numbers are rounded to relevant digits. This shows that the origin is an isolated singularity of $F$. 
\end{ex}

This example shows that if we want to establish an equivalence between isolated singularities and strong partial non-degeneracy for semiholomorphic polynomial we have to exclude radially weighted homogeneous polynomials for which $\tfrac{\partial\arg(v_j(t))}{\partial t}$ has roots that are local maxima or minima of $\tfrac{\partial\arg(v_j(t))}{\partial t}$. We want to generalize this condition to semiholomorphic polynomials that are not radially weighted homogeneous.

We mentioned above that the first author proved in \cite{Bode2019} that an isolated singularity is equivalent to the roots of $g$ being distinct and $\tfrac{\partial\arg(v_j(t))}{\partial t}$ having no zeros. The calculation in that proof can be rewritten in terms of $s_{3,f}$ as follows. Let $f$ be a radially weighted homogeneous semiholomorphic polynomial with a weakly isolated singularity and let $g$ be as in Eq.~\eqref{resc}. As above we write $v_j(t)$ for the critical values of $g(\cdot,\rme^{\rmi t})$. We also write $c_j(t)$ for the critical points of $g(\cdot,\rme^{\rmi t})$, i.e., the roots of $\tfrac{\partial g}{\partial u}$, with $g(c_j(t),\rme^{\rmi t})=v_j(t)$. Then $f_u(r^kc_j(t),r\rme^{\rmi t})=0$ and $s_{3,f}(r^kc_j(t),r\rme^{\rmi t})$ is a positive multiple of $\tfrac{\partial\arg(v_j(t))}{\partial t}$, where $k=\tfrac{p_1}{p_2}$ is determined by the weight vector $P=(p_1,p_2)$ associated with the unique 1-face of $f$.

Therefore, if we have a semiholomorphic polynomial $f$ whose Newton boundary has several 1-faces, then the function that plays the role of $\tfrac{\partial\arg(v_j(t))}{\partial t}$ in the radially weighted homogeneous case should be something like the $P_i$-part of $s_{3,f}$, where $P_i$, $i=1,2,\ldots,N$, are the weight vectors associated with the 1-faces of $f$. Guided by this analogy we formulate the following two properties.

Let $f:\mathbb{C}^2\to \mathbb{C}$  be a semiholomorphic polynomial satisfying
\begin{itemize}
\item[(S-i)] the mixed polynomial $f_u$ is inner non-degenerate,
\item[(S-ii)] if $t_*$ is a local extremum (minimum or maximum) of $((s_{3,f})_{P_i})_i(u(t),t)$ for any $P_i$ and parametrization $(u(t),t)$\footnote{By \cite{AraujoBodeSanchez} such parametrization exists if $f_u$ is inner non-degenerate.} of the roots of $((f_u)_{P_i})_i$, then $(s_{3,f})_{P_i}(u(t_*),t_*)\neq 0$.
\end{itemize}

\begin{prop}\label{prop:strongboundary} Let $f$ be a semiholomorphic polynomial that satisfies the above conditions (S-i) and (S-ii). Then $f$ has an isolated singularity if and only if $f$ is strongly partially non-degenerate. 
\end{prop}
\begin{proof}
($\Leftarrow$) Follows from Proposition~\ref{strong-isolated3}.
\vspace{0.3cm}

($\Rightarrow$) As usual we denote the weight vectors associated with the 1-face of the Newton boundary of $f$ by $P_i$, $i=1,2,\ldots,N$. First, consider a weight vector $P$ that is different from all $P_i$, so that its corresponding face is a vertex $\Delta$ of the Newton boundary. Since $f$ is semiholomorphic, it is of the form $f_P(u,\rme^{\rmi t},r\rme^{-\rmi t})=f_{\Delta}(u,r\rme^{\rmi t},r\rme^{-\rmi t})=u^\mu r^\nu\Phi(\rme^{\rmi t})$ for some trigonometric polynomial $\Phi(\rme^{\rmi t})$ and natural numbers $\mu,\nu$. Then the inner non-degeneracy of $f_u$ implies that $\Phi$ has no zeros, which implies that $(f_u)_P$ has no zeros in $(\C^*)^2$ and $(f_u(0,v))_P$ has no zeros in $(\C^*)^2$, so that $f$ satisfies SPND-(i) and SPND-(ii) for every weight vector $P$ that is not $P_i$ for some $i$.

For the remaining weight vectors $P_i$ consider the decompositions of $f_u$ and $s_{3,f}$ associated with $P_i$ as in Eq.~\eqref{decomp}
\begin{align*}
f_u(u,r\rme^{\rmi t},r\rme^{-\rmi t}) &=r^{\tfrac{d(P_i;f_u)}{p_{i,2}}} (f_u)_i \left( \frac{u}{r^{k_i}} ,r,t \right)  \\
s_{3,f}(u,r\rme^{\rmi t},r\rme^{-\rmi t}) &=r^{\tfrac{d(P_i;s_{3,f})}{p_{i,2}}} (s_{3,f})_i \left( \frac{u}{r^{k_i}} ,r,t \right). 
\end{align*}  
Since $f_u$ is inner Newton non-degenerate, we can find a parametrization of the roots of $f_u$, $(r^{k_i}u(r,t),r\rme^{\rmi t})$, where $(u(r,t),r\rme^{\rmi t})$ is a root of $(f_u)_i$, $r<r_0$ and $t \in [0,2\pi]$, $r_0>0$ small enough, by \cite[Proposition 4.5]{AraujoBodeSanchez}. Note that $\lim_{r\to 0} u(r,t)=u(0,t)$ is the $u$-coordinate of a root of the family of univariate polynomials $((f_u)_{P_i})_i(\cdot,0,t)$.
We know that one inequality  
\begin{align}
&s_{3,f}(r^ku(r,t),r\rme^{\rmi t})=r^{\tfrac{d(P_i;s_{3,f})}{p_{i,2}}} (s_{3,f})_i ( u (r,t) ,r,t ) >0 \text{ or }  \label{semeq1suf}\ \\   &s_{3,f}(r^ku (r,t),r\rme^{\rmi t},r\rme^{-\rmi t})=r^{\tfrac{d(P_i;s_{3,f})}{p_{i,2}}} (f_u)_i (u(r,t) ,r,t ) <0 \label{semeq2suf}
\end{align}
holds for $(r,t) \in (0,r_0)\times[0,2\pi]$, with $r_0>0$ small enough. Otherwise the Intermediate Value Theorem gives us a critical point of $f$ arbitrarily close to the origin. Therefore, on the limit points $r=0$, 
$$(s_{3,f})_i(u (0,t),0,t)= ((s_{3,f})_{P_i})_i(u (0,t),0,t) \geq 0$$
or 
$$(s_{3,f})_i(u (0,t),0,t)= ((s_{3,f})_{P_i})_i(u (0,t),0,t) \leq 0$$ and by the condition (S-ii) $(u(0,t),0,t)$ is not a common solution of  $((f_u)_{P_i})_i$  and $((s_{3,f})_{P_i})_i$ for any $t$. Therefore, $f$ is strongly partially non-degenerate for $P_i$. 
\end{proof}
We know from Theorem~\ref{strong-isolated3} that every strongly partially non-degenerate semiholomorphic polynomial has an isolated singularity. Proposition~\ref{prop:strongboundary} tells us how far this condition is from being a complete characterisation of semiholomorphic polynomials with isolated singularity. It follows that a semiholomorphic polynomial that has an isolated singularity, but that is not strongly partially non-degenerate, does not satisfy (S-i) or (S-ii). In the following we present a new perspective on condition (S-ii), which should make it easier to understand.
\vspace{0.3cm}

Let $f$ be a semiholomorphic polynomial that is not a radially weighted homogeneous polynomial and has a unique 1-face, i.e., $f \neq f_{\Gamma}=f_P$, for some weight vector $P$. Consider the critical values  $v_j(t),\, j=1,2,\dots,\ell$, of  $g_1(\cdot,\rme^{\rmi t})$, where $g_1$ is the g-polynomial associated to $P=P_1$. As in the radially weighted homogeneous case, we can (under small additional assumptions) show that $((s_{3,f})_{P_1})_1(r^kc_j(t),t)$ is a positive multiple of $\tfrac{\partial\arg(v_j(t))}{\partial t}$. In particular, it is 0 if and only if $\tfrac{\partial\arg(v_j(t))}{\partial t}=0$ and always has the same sign as $\tfrac{\partial\arg(v_j(t))}{\partial t}$. Then (S-ii) says that for every $t$ where $\frac{\partial \arg(v_j(t))}{\partial t}=0$ we have that 
$\frac{\partial^2\arg(v_j(t))}{\partial t}\neq0$. Figure~\eqref{suf1} shows segments of the graph of a function $\tfrac{\partial \arg(v_j(t))}{\partial t}$ that does not have this property. At $t=t_1$ the function has a zero with $\frac{\partial^2\arg v_j}{dt}(t_1)=0$, while at $t=t_2$ the function evaluates to zero with  $\frac{\partial^2\arg v_j}{dt}(t_2)\neq0$. These two cases correspond to a tangential and a transverse intersection of the graph with the $t$-axis, respectively. 
\begin{figure}[H]\label{suf1}
    \centering
    \includegraphics[height=4.0cm]{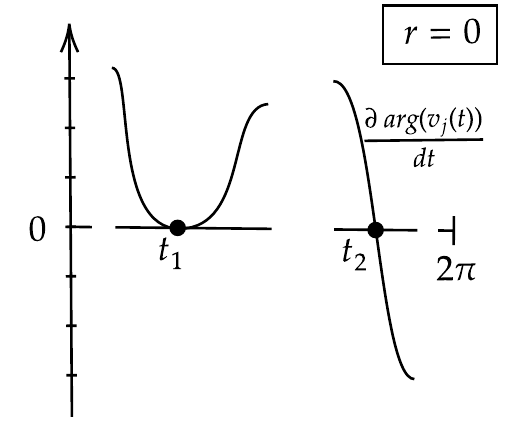}
    \caption{Parts of an example graph of $\frac{\partial \arg(v_j(t))}{dt}$.}
    \label{suf1}
\end{figure}
 \begin{figure}[H]
\begin{subfigure}[b]{.5\textwidth}
  \centering
  \includegraphics[height=0.5cm,angle=0,width=0.8\linewidth]{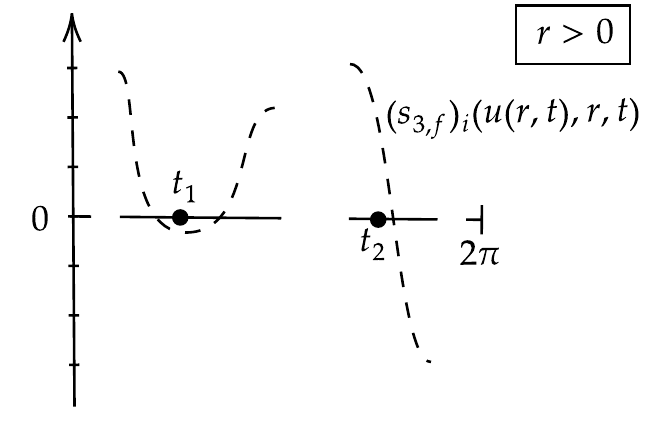}
  \caption{}
  \label{suf2}
\end{subfigure}
\begin{subfigure}[b]{.5\textwidth}
  \centering
  \includegraphics[height=0.5cm,angle=0,width=0.8\linewidth]{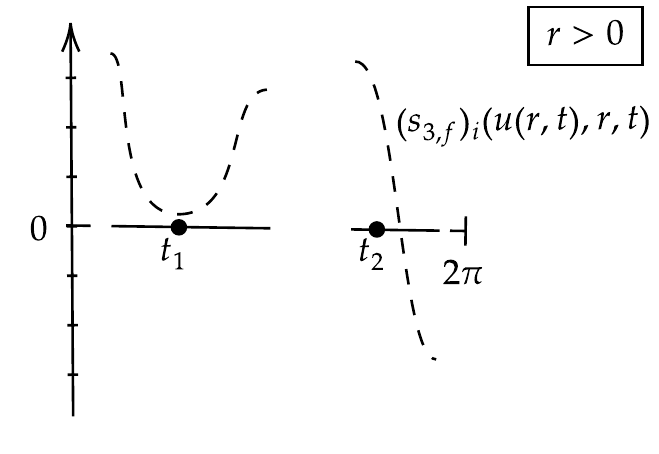}
  \caption{}
  \label{suf3}
\end{subfigure}
  \caption{Possible graphs of $(s_{3,f})_i(u(r,t), r,t)$ as functions of $t$ for a fixed small positive value of $r$.}
\label{suf}
\end{figure}
We know that as $r$ approaches zero $(s_{3,f})_1(u(r,t),r,t)$ converges to the function $((s_{3,f})_{P_1})_1(u(0,t),t)=((s_{3,f})_{P_1})_1(r^kc_j(t),t)$, which is a positive multiple of $\tfrac{\partial\arg(v_j(t))}{\partial t}$.  Therefore, for any fixed small positive value of $r$ the function $(s_{3,f})_1(u(r,t),r,t)$ still has a root in a neighborhood of $t=t_2$ and this root still arises as a transverse intersection of the graph of the function and the $t$-axis. In a neighborhood of $t=t_1$ however, the function might have two (or more) zeros as depicted in Figure~\eqref{suf2} or no zeros at all as in Figure~\eqref{suf3}. In other words, the only roots of $\partial\arg(v_j(t))/\partial t$ that do not necessarily correspond to roots of $(s_{3,f})_i$ are roots that are also local maxima or minima.

\vspace{0.3cm}
Note that since $f$ is semiholomorphic, the roots of $s_{3,f}(r^{k}u(r,t),r\rme^{\rmi t})$ correspond exactly to the critical points of $f$. Furthermore, the roots of the function $s_{3,f}(r^{k}u(r,t),r\rme^{\rmi t})$ also directly correspond to the roots of $(s_{3,f})_1(u(r,t),r,t)$. Therefore, Proposition~\ref{prop:strongboundary} guarantees that if $f$ has an isolated singularity, then $f_P$ cannot admit singularities like the one corresponding to $t=t_2$ in the example. In other words, this type of singularity at $t_2$ associated to $f_P$ implies a non-isolatedness of the singularity at the origin. Thus we can see Proposition~\ref{prop:strongboundary} as a criterion not just to say when $f$ has an isolated singularity but also when $f$ has a non-isolated singularity at the origin.  

\vspace{0.3cm}
Generalizing Theorem~\ref{prop:strongboundary} to the general mixed case comes with a difficulty. The singular set of a mixed function is defined by 3 equations, $s_{1,f},s_{2,f}$ and $s_{3,f}$, while the singular set of a semiholomorphic function only requires two (see Eq.~\eqref{criticalpointsem}). This forces us to impose an extra condition on the face functions of partial derivatives of $f$ as follows.

\vspace{0.3cm}

Let $f$ be a mixed polynomial. Let $P_i$, $i=1,2,\ldots,N$, denote the weight vectors associated with the 1-faces of $\Gamma_{s_{1,f}}$. Consider the following conditions.
\begin{itemize}
\item[(M-i)] The mixed polynomial $s_{1,f}$ is nice and inner Newton non-degenerate, see \cite{AraujoBodeSanchez}.   
\item[(M-ii)] For any $P_i$, $i\in\{1,2,\ldots,N\}$, $\tau_0 \in [0,2\pi]$ with $((s_{j,f})_{P_i})_i(u(\tau_0), t(\tau_0))=0, \ j=1,2,$ the point $ (u(\tau_0), t(\tau_0))$ is not a local extremum of the function $((s_{j,f})_{P_i})_i(u(\tau), t(\tau))$, $j=1,2,$ where $(u(\tau), t(\tau))\ $ is a local parametrization of the roots of $((s_{1,f})_{P_i})_i$.\footnote{By \cite{AraujoBodeSanchez} such parametrization exists if $s_{1,f}$ is inner non-degenerate and has a nice Newton boundary.}
\item[(M-iii)] For any $P_i$, $i\in\{1,2,\ldots,N\}$, the systems $$(f_v(u,0))_{P_i}=(f_{\bar{v}}(u,0))_{P_i} =(f_u(u,0))_{P_i}=(f_{\bar{u}}(u,0))_{P_i}=0,$$ $$(f_v(0,v))_{P_i}=(f_{\bar{v}}(0,v))_{P_i} =(f_u(0,v))_{P_i}=(f_{\bar{u}}(0,v))_{P_i}=0$$ and $$(f_v)_{P_i}=(f_{\bar{v}})_{P_i} =(f_u)_{P_i}=(f_{\bar{u}})_{P_i}=0$$  have no solution in $(\C^*)^2$. 
\end{itemize}
\begin{prop}\label{prop:strongboundary2} Let $f$ be a mixed polynomial then satisfies the conditions (M-i), (M-ii) and (M-iii). Then $f$ has an isolated singularity at the origin if and only if $f$ is strongly partially non-degenerate. 
\end{prop}
\begin{proof}
($\Leftarrow$) Follows directly from Proposition~\ref{strong-isolated3}.

\vspace{0.3cm}
($\Rightarrow$) 
Let $P$ be a weight vector such that the corresponding face in $\Gamma_{s_{1,f}}$ is a non-extreme vertex. Since $s_{1,f}$ has a nice Newton boundary, it follows that $(s_{1,f})_P$ has no roots in $(\C^*)^2$. Therefore, $f$ satisfies SPND-(ii) for $P$. It remains to prove SPND-(ii) for weight vectors $P$ that are associated with a 1-face or an extreme vertex of $\Gamma_{s_{1,f}}$ and to prove SNPD-(i). We split the proof in two cases, the first discusses $P$ associated with a 1-face and the second contains the case of extreme vertex as well as the proof of the condition SPND-(i).

We index the 1-faces of $\Gamma_{s_{1,f}}$ by $i\in\{1,2,\ldots,N\}$ with weight vectors $P_i$ as usual. Let $P_i=(p_{i,1},p_{i,2})$ and $k_i=\tfrac{p_{i,1}}{p_{i,2}}$.\\

 \noindent \underline{Case 1: $P$ corresponds to a 1-face of $\Gamma_{s_{1,f}}$}\\
  
In this case, $P=P_i$ for some $i\in\{1,2,\ldots,N\}$.
 
 Consider the decomposition from Eq.~\eqref{decomp} of each $s_{1,f}$, $s_{2,f}$ and $s_{3,f}$ associated to $P_i$, i.e.,  
\begin{align*}
s_{1,f}(u,r\rme^{\rmi t}) &=r^{\tfrac{d(P_i;s_{1,f})}{p_{i,2}}} (s_{1,f})_i \left( \frac{u}{r^{k_i}} ,r,t \right)  \\
s_{2,f}(u,r\rme^{\rmi t}) &=r^{\tfrac{d(P_i;s_{2,f})}{p_{i,2}}} (s_{2,f})_i \left( \frac{u}{r^{k_i}} ,r,t \right)  \\
s_{3,f}(u,r\rme^{\rmi t}) &=r^{\tfrac{d(P_i;s_{3,f})}{p_{i,2}}} (s_{3,f})_i \left( \frac{u}{r^{k_i}} ,r,t \right) 
\end{align*}  
Since $s_{1,f}$ is nice and inner non-degenerate, we can find a parametrization $(r^{k_i}u(r,\tau),r\rme^{\rmi t(r,\tau)})\in \C^{*2}$ of the roots of $(s_{1,f})$, where $(u(r,\tau),r\rme^{\rmi t(r,\tau)})$ is a root of $(s_{1,f})_i$ (cf. \cite[Theorem 4.6]{AraujoBodeSanchez}). Note that $\lim_{r\to 0} (u(r,\tau),r, t(r,\tau))=(u(0,\tau),0,t(0,\tau))\in \C^{*2}$ is a root of $((s_{1,f})_{P_i})_i$. In particular, the roots of $(s_{1,f})_{P_i}$ are given by $(r^{k_i}u(0,\tau),r\rme^{\rmi t(0,\tau)})$.

Suppose that some inequality of 
\begin{align}
&s_{2,f}(r^{k_i}u(r,\tau),r\rme^{\rmi t(r,\tau)})=r^{\tfrac{d(P_i;s_{2,f})}{p_{i,2}}} (s_{2,f})_i ( u(r,\tau),r, t(r,\tau) ) >0 \label{eq1suf}\\ 
  &s_{2,f}(r^{k_i}u(r,\tau),r\rme^{\rmi t(r,\tau)})=r^{\tfrac{d(P_i;s_{2,f})}{p_{i,2}}} (s_{2,f})_i (u (r,\tau) ,r,t(r,\tau) ) <0  \label{eq2suf}\ \\
&s_{3,f}(r^{k_i}u(r,\tau),r\rme^{\rmi t(r,\tau)})=r^{\tfrac{d(P_i;s_{3,f})}{p_{i,2}}} (s_{3,f})_i ( u (r,\tau) ,r,t(r,\tau) ) >0  \label{eq3suf}\ \\   
&s_{3,f}(r^{k_i}u(r,\tau),r\rme^{\rmi t(r,\tau)})=r^{\tfrac{d(P_i;s_{3,f})}{p_{i,2}}} (s_{3,f})_i (u (r,\tau) ,r,t(r,\tau)) <0  \label{eq4suf}
\end{align}
holds for $(r,\tau) \in (0,r_0)\times(\tau_0-\epsilon, \tau_0+\epsilon)$, with $r_0>0$ and $\epsilon>0$ small enough, for instance Eq.~\eqref{eq1suf}. By condition (M-ii) we obtain on the limit point $r=0$ that 
$$(s_{2,f})_i(u(0,\tau),0,t(0,\tau))= ((s_{2,f})_{P_i})_i(u(0,\tau),0,t(0,\tau)) > 0.$$ Therefore $(u(0,\tau),0,t(0,\tau))$ is not a common zero of  $((s_{1,f})_{P_i})_i,\ ((s_{2,f})_{P_i})_i$ and $((s_{3,f})_{P_i})_i$ for any $\tau \in (\tau_0-\epsilon,\tau_0+\epsilon)$. Since at $r=0$ these functions are equal to the $g$-polynomial associated with $(s_{1,f})_{P_i}$, $(s_{2,f})_{P_i}$ and $(s_{3,f})_{P_i}$, respectively, we find that $(u(0,\tau),\rme^{\rmi t(0,\tau)})$ is not a common zero of these $g$-polynomials, which by definition means that $(r^{k_i}u(0,\tau),r\rme^{\rmi t(0,\tau)})$ is not a common root of $(s_{1,f})_{P_i}$, $(s_{2,f})_{P_i}$ and $(s_{3,f})_{P_i}$. But since these are by construction all the roots of $(s_{1,f})_{P_i}$, there are no common zeros at all.

On the other hand, suppose that there are no $r_0,\epsilon>0$ small enough such that some inequality \eqref{eq1suf}-\eqref{eq4suf} holds. Then by the Intermediate Value Theorem there are for every $r>0$, $\varepsilon>0$, tuples $(r_1,\tau_1’)$ and $(r_1’,\tau_1’)$ with $r_1,r_1'<r$ and $\tau_1,\tau_1'\in[\tau_0-\varepsilon,\tau_0+\varepsilon]$ such that 
\begin{equation}\label{eq1contraditioniso}
s_{1,f}(r_1^{k_i}u (r_1,\tau_1),r_1\rme^{\rmi t(r_1,\tau_1)})=s_{2,f}(r_1^{k_i}u  (r_1,\tau_1),r_1\rme^{\rmi t(r_1,\tau_1)})=0, 
\end{equation}
with $$f_u (r_1^{k_i}u (r_1,\tau_1),r_1\rme^{\rmi t(r_1,\tau_1)})=f_{\bar{u}}(r_1^{k_i}u (r_1,\tau_1),r_1\rme^{\rmi t(r_1,\tau_1)})=0$$ and  
\begin{equation}\label{eq2contraditioniso}
s_{1,f}((r’_1)^{k_i}u(r’_1,\tau’_1),r’_1\rme^{\rmi t(r’_1,\tau’_1)})=s_{3,f}((r’_1)^{k_i}u(r’_1,\tau’_1),r’_1\rme^{\rmi t(r’_1,\tau’_1)})=0, 
\end{equation}
 with $$f_v ((r’_1)^{k_i}u(r’_1,\tau’_1),r’_1\rme^{\rmi t(r’_1,\tau’_1)})=f_{\bar{v}}((r’_1)^{k_i}u(r’_1,\tau’_1),r’_1\rme^{\rmi t(r’_1,\tau’_1)})=0.$$
Otherwise, Eq.~\eqref{eq1contraditioniso} and Eq.~\eqref{eq2contraditioniso} imply $s_{3,f}(r_1^{k_i}u (r_1,\tau_1),r_1\rme^{\rmi t(r_1,\tau_1)})=0$ and $s_{2,f}((r’_1)^{k_i}u(r’_1,\tau’_1),r’_1\rme^{\rmi t(r’_1,\tau’_1)})=0$, respectively,  either of which is a contradiction to the isolatedness  at the origin. Indeed, suppose that $f_u (u_*,v_*)\neq0$. Then $s_{2,f}(u_*,v_*)=0$ implies $|f_u (u_*,v_*)/ \bar{f}_{\bar{u}}(u_*,v_*)|=1$ and $s_{1,f}(u_*,v_*)=0$ implies  $f_v(u_*,v_*)=(f_u (u_*,v_*)/ \bar{f}_{\bar{u}}(u_*,v_*))\cdot \bar{f}_{\bar{v}}(u_*,v_*)$, which is equivalent in this case to $s_{3,f}(u_*,v_*)=0$. 
 \vspace{0.3cm}
 
 Thus we can construct sequences $$(r_n^{k_i}u (r_n,\tau_n),r_n\rme^{\rmi t(r_n,\tau_n)}) \text{ and } ((r’_n)^{k_i}u(r’_n,\tau’_n),r’_n\rme^{\rmi t(r’_n,\tau’_n)}),$$ with sequences $(r_n,\tau_n)$ and $(r’_n,\tau’_n)$ $\in (0,1)\times (\tau_0-\epsilon,\tau_0+\epsilon)$ converging to $(0,\tau_0)$ and such that 
 $$f_u(r_n^{k_i}u (r_n,\tau_n),r_n\rme^{\rmi t(r_n,\tau_n)})=r_n^{\tfrac{d(P_i;f_u)}{p_{i,2}}}(f_u)_i(u(r_n,\tau_n),r_n,t(r_n,\tau_n))=0,$$
 $$f_{\bar{u}}(r_n^{k_i}u (r_n,\tau_n),r_n\rme^{\rmi t(r_n,\tau_n)})=r_n^{\tfrac{d(P_i;f_{\bar{v}})}{p_{i,2}}}(f_{\bar{u}})_i(u(r_n,\tau_n),r_n,t(r_n,\tau_n))=0 $$  and 
 $$f_v((r’_n)^{k_i}u(r’_n,\tau’_n),r’_n\rme^{\rmi t(r’_n,\tau’_n)})= (r’_n)^{\tfrac{d(P_i;f_u)}{p_{i,2}}}(f_v)_i(u(r’_n,\tau’_n),r’_n, t(r’_n,\tau’_n))=0, $$
$$ f_{\bar{v}}((r’_n)^{k_i}u(r’_n,\tau’_n),r’_n\rme^{\rmi t(r’_n,\tau’_n)})= (r_n)^{\tfrac{d(P_i;f_{\bar{u}})}{p_{i,2}}}(f_{\bar{v}})_i(u(r’_n,\tau’_n),r’_n, t(r’_n,\tau’_n))=0 .$$ 
 Then  
\begin{align*}
(f_u)_i(u(r_n,&\tau_n),r_n,t(r_n,\tau_n))=(f_{\bar{u}})_i(u(r_n,\tau_n),r_n,t(r_n,\tau_n))=\\&=(f_v)_i(u(r’_n,\tau’_n),r’_n, t(r’_n,\tau’_n))=(f_{\bar{v}})_i(u(r’_n,\tau’_n),r’_n, t(r’_n,\tau’_n))=0.  
\end{align*}
Therefore, taking the limit $(r_n,\tau_n) \to (0,\tau_0)$ and $(r’_n,\tau’_n)\to (0,\tau_0)$, we have 
 $$((f_u)_{P_i})_i(u(0,\tau_0),0,t(0,\tau_0))=((f_{\bar{u}})_{P_i})_i(u(0,\tau_0),0,t(0,\tau_0))=0$$  and $$((f_v)_{P_i})_i(u(0,\tau_0),0,t(0,\tau_0))=((f_{\bar{v}})_{P_i})_i(u(0,\tau_0),0,t(0,\tau_0))=0,$$
  which is a contradiction to (M-iii). Therefore, $(s_{1,f})_{P_i},\ (s_{2,f})_{P_i}$ and  $(s_{3,f})_{P_i}$ have no common zeros. Thus $f$ satisfies SPND-(ii).\\

 \noindent \underline{Case 2: $P=\tfrac{p_{1}}{p_{2}}$ with $\tfrac{p_1}{p_2}>k_1$ or $\tfrac{p_1}{p_2}<k_N$}.  
 
\noindent  We are going to prove that $f$ satisfies SPND-(ii) for $P$ with $\frac{p_{1}}{p_{2}} >k_1$ and also that 
\begin{equation}\label{equivtheoremeq0}
(s_{1,f}(0,v))_P=(s_{2,f}(0,v))_P=(s_{3,f}(0,v))_P=0
\end{equation}
has no solution in $(\C^*)^2$. The other cases are very similar. 

Since $s_{1,f}$ is IND, we may apply the same arguments as in Case 1 to a parametrization of the zeros of $(s_{1,f})_1$ in $\C^2\setminus \{v=0\}$. It follows that $(s_{1,f})_{P_i}=(s_{2,f})_{P_i}=(s_{3,f})_{P_i}=0$ has no solution in $\C^2\setminus \{v=0\}$ for any $i\in\{1,2,\ldots,N\}$. In particular,
\begin{equation}\label{equivtheoremeq1}
(s_{1,f})_{P_1}(0,v)=(s_{2,f})_{P_1}(0,v)=(s_{3,f})_{P_1}(0,v)=0
\end{equation}
 has no solution for any $v\in \C^*$. 
 It is clear that 
 \begin{equation}\label{equivtheoremeq2}
  (s_{1,f})_{P_1}(0,v)=
 \begin{cases}
  (s_{1,f}(0,v))_{P_1} & \text{if $s_{1,f}$ is $v$-convenient }\\ 
   \ \ \ \ \ \ 0                 & \text{otherwise.}
 \end{cases}
\end{equation}
For $s_{j,f},\,j=2,3$, we denote by $k_{1,j}$ the number associated to the first 1-face of $\Gamma_{s_{2,f}}$ and $\Gamma_{s_{3,f}}$, respectively. That is if $P_i'$, $i\in\{1,2,\ldots,N'\}$, and $P_i''$, $i\in\{1,2,\ldots,N''\}$, are the weight vectors corresponding to the 1-faces of $s_{2,f}$ and $s_{3,f}$, respectively, with $P_1'=(p_1',p_2')$ and $P_1''=(p_1'',p_2'')$, then $k_{1,2}=\tfrac{p_1'}{p_2'}$ and $k_{1,3}=\tfrac{p_1''}{p_2''}$. Thus,
 \begin{equation}\label{equivtheoremeq3}
  (s_{2,f})_{P_1}(0,v)=
 \begin{cases}
  (s_{2,f}(0,v))_{P_1} & \text{if $k_{1} \geq  k_{1,2}$}\\ 
   \ \ \ \ \ \ 0                 & \text{otherwise,}
 \end{cases}
\end{equation}
and 
 \begin{equation}\label{equivtheoremeq4}
  (s_{3,f})_{P_1}(0,v)=
 \begin{cases}
  (s_{3,f}(0,v))_{P_1}  & \text{if $k_{1} \geq  k_{1,3}$}\\ 
   \ \ \ \ \ \ 0                 & \text{otherwise.}
 \end{cases}
\end{equation}
Since system~\eqref{equivtheoremeq1} does not have solutions with $v\in \C^*$, it follows from Eqs.\eqref{equivtheoremeq2}-\eqref{equivtheoremeq4} that a lack of solution of  \eqref{equivtheoremeq1} implies a lack of solution of 
\begin{equation}\label{eq:another}
(s_{1,f}(0,v))_{P_1}=(s_{2,f}(0,v))_{P_1}=(s_{3,f}(0,v))_{P_1}=0
\end{equation} 
in $(\C^*)^2$.

More precisely, since there are no solutions to Eq.~\eqref{equivtheoremeq1}, not all of $(s_{1,f})_{P_1}(0,v)$, $(s_{2,f})_{P_1}(0,v)$ and $(s_{3,f})_{P_1}(0,v)$ are constant 0. Let $J\subset\{1,2,3\}$ be the set of indices $j$ with $(s_{j,f})_{P_1}(0,v)\not\equiv0$. Then the lack of solutions to Eq.~\eqref{equivtheoremeq1} implies that there are no common zeros of $(s_{j,f})_{P_1}(0,v)=(s_{j,f}(0,v))_{P_1}$, $j\in J$, and in particular no solution to Eq.~\eqref{eq:another} in $(\C^*)^2$.

Since for any $P$ and any $i=1,2,3$, we have $(s_{i,f}(0,v))_{P_1}=(s_{i,f}(0,v))_{P}$, it follows that the system \eqref{equivtheoremeq0} has no solution in $(\C^{*})^2$. Making analogous arguments we get that for any $P$
\begin{equation}
(s_{1,f}(u,0))_P=(s_{2,f}(u,0))_P=(s_{3,f}(u,0))_P=0
\end{equation}
 has no solution in $(\C^*)^2$. In other words, $f$ satisfies SPND-(i).
    
Note that if $(s_{i,f})_{P_1}(0,v) \not\equiv 0$, for some $i\in \{1,2,3\},$ then we have $(s_{i,f})_{P_1}(0,v)= (s_{i,f}(0,v))_{P_1}$. Moreover, $s_{i,f}$ is $v$-convenient. Thus, we find by Eqs.\eqref{equivtheoremeq2}-\eqref{equivtheoremeq4} that
\begin{equation}\label{equivtheoremeq5}
 (s_{i,f})_{P_1}(0,v)= (s_{i,f}(0,v))_{P_1}=(s_{i,f})_P(u,v), 
 \end{equation}     
 for all $P=(p_1,p_2)$ with $\frac{p_1}{p_2}>k_{1}$.
Again, a lack of solution of \eqref{equivtheoremeq1} implies that for some $i \in \{1,2,3\}$ Eq.\eqref{equivtheoremeq5} holds and thus for any $P$ with $\frac{p_1}{p_2}>k_{1}$
\begin{equation}\label{equivtheoremeq6} 
(s_{1,f})_P(u,v)=(s_{2,f})_P(u,v)=(s_{3,f})_P(u,v)=0
\end{equation}
has no solution in $(\C^{*})^2$. 
Therefore, $f$ satisfies SPND-(ii) for $P$ with $\frac{p_1}{p_2}>k_{1}$. We can prove with appropriate changes that $f$ satisfies SPND-(ii) for $P$ with $\frac{p_1}{p_2}<k_{N}$. Hence, $f$ is strongly partially non-degenerate.
\end{proof}

As a consequence of the proof of Theorem~\ref{prop:strongboundary} we have the following sufficient condition to prove that a mixed polynomial has a non-isolated singularity. 
 \begin{corolario}\label{non-isolated}
Let $f$ be a mixed polynomial and $p=(u_*,v_*)$ a solution of the system $(s_{1,f})_{P}=(s_{2,f})_{P}=(s_{3,f})_{P}=0$, for some weight vector $P$. If there exists a local parametrization of the roots of $(s_{1,f})_{P}$ around $p$, which does not satisfy (M-ii) nor (M-iii), then $f$ has a non-isolated singularity at the origin.   
  \end{corolario}

\section{THE STRONG MILNOR CONDITION}\label{section4}
  Recall that $f$ satisfies {\it the strong Milnor condition} if there is a positive number $\rho_0>0$ such that
\begin{equation}\label{arg}
\arg(f):=\frac{f}{|f|} : S^{3}_{\rho} \setminus V_{f}\to S^1
\end{equation}  is a locally trivial fibration for every radius $\rho \leq \rho_0$. The Milnor set $M_{\arg f}$ of the mapping $\arg f:\C^2 \setminus V_f \to S^1$ is defined as 
\begin{equation}\label{milnorsetargf}
M_{\arg f}=\{z \in \C^n \setminus V_f\, |\, \exists \lambda \in \R, \lambda z = \rmi (f \overline{df}- \bar{f} \bar{d}f) \},
\end{equation}
where $\rmd f=\left(\tfrac{\partial f}{\partial u},\tfrac{\partial f}{\partial v}\right)$ and $\bar{d}f=\left(\tfrac{\partial f}{\partial\bar{u}},\tfrac{\partial f}{\partial \bar{v}}\right)$. By \cite[Theorem 2.2]{AraujoTibar} when $f$ has a weakly isolated singularity, the strong Milnor condition is equivalent to $U\cap M_{\arg f}= \emptyset$ for sufficiently small neighbourhoods $U$ of the origin in $0\mathbb{R}^4$.

The fact that strongly Newton non-degenerate mixed polynomials with convenient Newton boundary satisfy the strong Milnor condition was proved in Theorem 33 in \cite{Oka2010}. Now we show that this is in fact true for the larger class of strongly partially non-degenerate mixed polynomials. For this we need several lemmas.
 
 


\begin{lemma}\label{lem:extremeSND}
Let $f : (\C^2,0) \to (\C,0)$ be a strongly inner non-degenerate mixed polynomial. Suppose that $f$ is not $v$-convenient. Then the corresponding extreme vertex $f_{\Delta_1}$ is strongly Newton non-degenerate. Likewise, if $f$ is not $u$-convenient, then $f_{\Delta_{N+1}}$ is strongly Newton non-degenerate. It follows that if $f$ is neither $u$-convenient nor $v$-convenient, then $f$ is strongly Newton non-degenerate.
\end{lemma}
\begin{proof}
Recall that if $f$ is SIND, but not $v$-convenient, there is a $n\in\mathbb{N}$ such that $(1,n)\in supp(f)$. Otherwise, all summands in all of $s_{1,f_{P_1}}$, $s_{2,f_{P_1}}$ and $s_{3,f_{P_1}}$ involve $u$ or $\bar{u}$, so that $s_{1,f_{P_1}}(0,v)=s_{2,f_{P_1}}(0,v)=s_{3,f_{P_1}}(0,v)=0$ for all $v\in\C$, which contradicts SIND-(i).

Since $f$ is not $v$-convenient, we have 
\begin{equation}\label{eq:p10v}
(f_{P_1})_v(0,v)=(f_{P_1})_{\bar{v}}(0,v)=0
\end{equation} 
for all $v\in \C$. Since there is a $n\in\mathbb{N}$ such that $(1,n)\in supp(f)$, the face function $f_{\Delta_1}$ of the extreme vertex $\Delta_1$ is of the form $B(v,\bar{v})u+\bar{u}C(v,\bar{v})$ for some polynomials $B(v,\bar{v})$ and $C(v,\bar{v})$. Thus $(f_{\Delta_1})_u=B(v,\bar{v})$ and $(f_{\Delta_1})_{\bar{u}}=C(v,\bar{v})$. Now suppose that $(u_*,v_*)\in (\C^*)^2$ is a critical point of $f_{\Delta_1}$. Then $s_{2,f_{\Delta_1}}(u_*,v_*)=0$ implies that $|B(v_*,\bar{v_*})|^2-|C(v_*,\bar{v_*})|^2=0$, which does not depend on $u_*$. Therefore, $s_{2,f_{\Delta_1}}(u,v_*)=0$ for all $u\in \C$, in particular for $u=0$. Now note that $s_{2,f_{\Delta_1}}(0,v_*)=s_{2,f_{P_1}}(0,v_*)$ and so $s_{2,f_{P_1}}(0,v_*)=0$. By Eq.~\eqref{eq:p10v} we also have $s_{1,f_{P_1}}(0,v_*)=0$ and $s_{3,f_{P_1}}(0,v_*)=0$. In other words, $(0,v_*)$ is a critical point of $f_{P_1}$, contradicting SIND-(i).

The proof for mixed polynomials that are not $u$-convenient follows the same reasoning. Now recall that SIND-(ii) implies that all face functions of $f$ except possibly $f_{\Delta_1}$ and $f_{\Delta_{N+1}}$ are strongly Newton non-degenerate. But we have shown that both of these are SND as well if $f$ is neither $u$-convenient nor $v$-convenient.
\end{proof}

A mixed polynomial $f:(\C^2,0) \to (\C,0)$ can be written as
\begin{equation}
f(z,\bar{z})=\sum_{\nu+\mu\in supp(f)}c_{\nu,\mu}z^\nu\bar{z}^\mu,
\end{equation}
where $z=(u,v)$, $\bar{z}=(\bar{u},\bar{v})$, $\nu=(\nu_1,\nu_2)$, $\mu=(\mu_1,\mu_2)$, $z^\nu=u^{\nu_1}v^{\nu_2}$ and $\bar{z}^\mu=\bar{u}^{\mu_1}\bar{v}^{\mu_2}$. For any positive weight vector $P=(p_1,p_2)$ we define
\begin{equation}
f_{(d,P)}(u,v):=\underset{p_1(\nu_1+\mu_1)+p_2(\nu_2+\mu_2)=d}{\sum_{\nu+\mu\in supp(f)}}c_{\nu,\mu}z^\nu\bar{z}^\mu,
\end{equation}
that is, $f_{(d,P)}$ is the sum of terms of $f$ whose radial degree with respect to $P$ is equal to $d$. It follows that $f_{(d,P)}=0$ if $d<d(P;f)$ and $f_{(d(P;f),P)}=f_P$ if $d=d(P;f)$.

\begin{lemma}\label{lem:w}
Let $f : (\C^2,0) \to (\C,0)$ be a mixed polynomial, $P$ be a positive weight vector. Let $\alpha$ be some non-zero complex number and $d=d(P;f)$. Define
\begin{align}
w_1(u,v)&=\rmi(\alpha(\overline{f_u})_P-\bar{\alpha}(f_{\bar{u}})_P)_{(d-p_1,P)}(u,v),\nonumber\\
w_2(u,v)&=\rmi(\alpha(\overline{f_v})_P-\bar{\alpha}(f_{\bar{v}})_P)_{(d-p_2,P)}(u,v).
\end{align}
Then the common zeros of $w_1$ and $w_2$ are critical points of $f_P$.
\end{lemma}
\begin{obs}
Observe that by \cite[Lemma 3.5]{AraujoBodeSanchez} we have $d(P;f_u)\geq d-p_1$ with equality if and only if $(f_u)_{(d-p_1,P)}=\left((f_u)_P\right)_{(d-p_1,P)}$ and likewise for derivatives with respect to the other variables $\bar{u}$, $v$ and $\bar{v}$. Therefore, $w_1$ and $w_2$ could equivalently be defined as
\begin{align}
w_1(u,v)&=\rmi(\alpha(\overline{f_u})-\bar{\alpha}(f_{\bar{u}}))_{(d-p_1,P)}(u,v),\nonumber\\
w_2(u,v)&=\rmi(\alpha(\overline{f_v})-\bar{\alpha}(f_{\bar{v}}))_{(d-p_2,P)}(u,v).
\end{align}
\end{obs}
\begin{proof}[Proof of Lemma~\ref{lem:w}]
By \cite[Lemma~3.5]{AraujoBodeSanchez} there are different cases to consider, depending on whether $f_P$ is semiholomorphic with respect to any of the variables $u$, $\bar{u}$, $v$ and $\bar{v}$. We have
\begin{align}
w_1(u,v)&=\begin{cases}
0 &\text{if }(f_P)_u=(f_P)_{\bar{u}}=0,\\
\rmi\alpha(\overline{f_u})_P(u,v) &\text{if }(f_P)_u\neq0,(f_P)_{\bar{u}}=0,\\
-\rmi\bar{\alpha}(f_{\bar{u}})_P(u,v)&\text{if }(f_P)_u=0,(f_P)_{\bar{u}}\neq0,\\
\rmi(\alpha(\overline{f_u})_P-\bar{\alpha}(f_{\bar{u}})_P)(u,v)  &\text{if }(f_P)_u\neq0,(f_P)_{\bar{u}}\neq0,\\
\end{cases}\\
w_2(u,v)&=\begin{cases}
0 &\text{if }(f_P)_v=(f_P)_{\bar{v}}=0,\\
\rmi\alpha(\overline{f_v})_P(u,v) &\text{if }(f_P)_v\neq0,(f_P)_{\bar{v}}=0,\\
-\rmi\bar{\alpha}(f_{\bar{v}})_P(u,v) &\text{if }(f_P)_v=0,(f_P)_{\bar{v}}\neq0,\\
\rmi(\alpha(\overline{f_v})_P-\bar{\alpha}(f_{\bar{v}})_P)(u,v) &\text{if }(f_P)_v\neq0,(f_P)_{\bar{v}}\neq0.\\
\end{cases}
\end{align}
Furthermore, we know from \cite[Lemma 3.5]{AraujoBodeSanchez} that if $f_P$ is not $\bar{x}$-semiholomorphic, i.e., $(f_P)_{x}\neq 0$, for some $x\in\{u,\bar{u},v,\bar{v}\}$, then $(f_P)_x=(f_x)_P$. This leads to
\begin{align}
w_1(u,v)&=\begin{cases}
0 &\text{if }(f_P)_u=(f_P)_{\bar{u}}=0,\\
\rmi\alpha\overline{(f_P)_u}(u,v) &\text{if }(f_P)_u\neq0,(f_P)_{\bar{u}}=0,\\
-\rmi\bar{\alpha}(f_P)_{\bar{u}}(u,v) &\text{if }(f_P)_u=0,(f_P)_{\bar{u}}\neq0,\\
\rmi(\alpha\overline{(f_P)_u}-\bar{\alpha}(f_P)_{\bar{u}})(u,v)  &\text{if }(f_P)_u\neq0,(f_P)_{\bar{u}}\neq0,\\
\end{cases}\\
w_2(u,v)&=\begin{cases}
0 &\text{if }(f_P)_v=(f_P)_{\bar{v}}=0,\\
\rmi\alpha\overline{(f_P)_v}(u,v) &\text{if }(f_P)_v\neq0,(f_P)_{\bar{v}}=0,\\
-\rmi\bar{\alpha}(f_P)_{\bar{v}}(u,v)  &\text{if }(f_P)_v=0,(f_P)_{\bar{v}}\neq0,\\
\rmi(\alpha\overline{(f_P)_v}-\bar{\alpha}(f_P)_{\bar{v}})(u,v) &\text{if }(f_P)_v\neq0,(f_P)_{\bar{v}}\neq0.\\
\end{cases}
\end{align}

Recall that if $f_P$ is $x$-semiholomorphic with $x\in\{u,\bar{u},v,\bar{v}\}$, then the equations that define its critical set simplify to $(f_P)_x=s_{j,f_P}=0$, where $j=2$ if $x\in\{v,\bar{v}\}$ and $j=3$ if $x\in\{u,\bar{u}\}$. 

Now let $(u_*,v_*)\in\C^2$ be such that $w_1(u_*,v_*)=w_2(u_*,v_*)=0$. Then if $f_P$ is $x$-semiholomorphic, the equation $w_i(u_*,v_*)=0$ implies $(f_P)_x(u_*,v_*)=0$, where $i=1$ if $x\in\{u,\bar{u}\}$ and $i=2$ if $x\in\{v,\bar{v}\}$. In particular, if $f_P$ depends exactly on one of $\{u,\bar{u}\}$, say $x$, and exactly on one of $\{v,\bar{v}\}$, say $y$, then $w_1$ and $w_2$ are non-zero multiples of $(f_x)_P=(f_P)_x$ and $(f_y)_P=(f_P)_y$, respectively. Hence common zeros of $w_1$ and $w_2$ are critical points of $f_P$.

If $f_P$ is neither $v$- nor $\bar{v}$-semiholomorphic, then the equation $w_2(u_*,v_*)=0$ implies that $(\alpha(\overline{f_v})_P-\bar{\alpha}(f_{\bar{v}})_P)(u_*,v_*)=0$. Since $\alpha\neq 0$, we obtain $s_{3,f_P}(u_*,v_*)=0$.  The analogous statement holds for $w_1(u_*,v_*)$ and $s_{2,f_P}$ if $f_P$ is not $u$-semiholomorphic and not $\bar{u}$-semiholomorphic. Thus, if $f_P$ depends on exactly three variables out of $\{u,\bar{u},v,\bar{v}\}$, then the common zeros of $w_1$ and $w_2$ are critical points of $f_P$.

Suppose now $f_P$ depends on all of $u,\bar{u},v,\bar{v}$. By the same argument as in the previous case we have $s_{2,f_P}(u_*,v_*)=s_{3,f_P}(u_*,v_*)=0$. Furthermore, $w_1(u_*,v_*)=w_2(u_*,v_*)=0$ implies that if $(f_x)_P(u_*,v_*)=0$ for any $x\in\{u,\bar{u},v,\bar{v}\}$, then $(f_{\bar{x}})_P(u_*,v_*)=0$ as well and so $s_{1,f_P}(u_*,v_*)=0$. We can thus assume that $(f_x)_P(u_*,v_*)\neq0$ for all $x\in \{u,\bar{u},v,\bar{v}\}$. It follows from $w_1(u_*,v_*)=w_2(u_*,v_*)=0$ that
\begin{equation}\label{eq:frac}
\frac{\overline{(f_P)_u}(u_*,v_*)}{(f_P)_{\bar{u}}(u_*,v_*)}=\frac{\bar{\alpha}}{\alpha}=\frac{\overline{(f_P)_v}(u_*,v_*)}{(f_P)_{\bar{v}}(u_*,v_*)}
\end{equation}
and so $s_{1,f_P}(u_*,v_*)=0$.

Lastly, suppose that $f_P$ depends neither on $u$ nor on $\bar{u}$. The argument for $v$ and $\bar{v}$ is analogous. Then $w_1(u,v)=(f_P)_u(u,v)=(f_P)_{\bar{u}}(u,v)=0$ for all $(u,v)\in\C^2$ and therefore both $s_{1,f_P}$ and $s_{2,f_P}$ are constant 0. The same reasoning as in the previous cases shows that a zero $(u_*,v_*)$ of $w_2$ is also a zero of $s_{3,f_P}$ and therefore a critical point of $f_P$. 

Thus in all possible cases $(u_*,v_*)$ is a critical point of $f_P$.
\end{proof}

We write $P_1,P_2,\ldots,P_N$ for the weight vectors associated with the compact 1-faces of the Newton boundary of $f$, ordered in the usual way \cite{Oka2010, AraujoBodeSanchez}. For $P_i=(p_{i,1},p_{i,2})$ we define $k_i=\tfrac{p_{i,1}}{p_{i,2}}$.  Let $\Delta(P_i;f)$, $i\in\{1,2,\ldots,N\}$, denote the set of integer lattice points that lie on the compact 1-face $\Delta(P_i)$ associated to $P_i$.\\

\begin{lemma}\label{lem:conv}
Let $f: (\C^2,0) \to (\C,0)$ be a $v$-convenient mixed polynomial with $(1,n)\notin\Delta(P_1;f)\cap \text{supp}(f)$ for all $n\in\mathbb{N}_0$ and let $P$ be a positive weight vector with $\tfrac{p_1}{p_2}>k_1$. Let $(w_1,w_2)$ be as in Lemma~\ref{lem:w} and $(u_*,v_*)\in \C^2$ be such that $w_1(u_*,v_*)=w_2(u_*,v_*)=0$. Then $(0,v_*)$ is a critical point of $f_{P_1}$. Likewise, $(u_*,0)$ is a critical point of $f_{P_N}$ if $f$ is $u$-convenient with $(n,1)\notin\Delta(P_N;f)\cap \text{supp}(f)$ for all $n\in\mathbb{N}_0$ and $k_N>\tfrac{p_1}{p_2}$.
\end{lemma}
\begin{proof}
We discuss the case of $f$ being $v$-convenient with $(1,n)\notin\Delta(P_1;f)\cap \text{supp}(f)$ for all $n\in\mathbb{N}_0$ and $\tfrac{p_1}{p_2}>k_1$. The case of $f$ being $u$-convenient and $k_N<\tfrac{p_1}{p_2}$ follows the same line of reasoning.

We have seen in Lemma~\ref{lem:w} that $(u_*,v_*)$ is a critical point of $f_P$. Since $f$ is convenient and $\tfrac{p_1}{p_2}>k_1$, the corresponding face $\Delta(P)$ is the extreme vertex $\Delta_1$ and so $f_P$ depends neither on $u$ nor on $\bar{u}$. Thus $(u,v_*)$ is a critical point of $f_P$ for any $u\in\C$, in particular for $u=0$. 

In this case, we have $(f_{P_1})_u(0,v)=(f_{P_1})_{\bar{u}}(0,v)=0$ for all $v\in\C$, which means that $s_{1,f_{P_1}}(0,v)=s_{2,f_{P_1}}(0,v)=0$ for all $v\in\C$. Moreover, $(f_{P_1})_v(0,v)=(f_P)_v(0,v)$ and $(f_{P_1})_{\bar{v}}(0,v)=(f_P)_{\bar{v}}(0,v)$ for all $v\in\C$, which means that $s_{3,f_P}(0,v)=0$ if and only if $s_{3,f_{P_1}}(0,v)=0$. Therefore, since $(0,v_*)$ is a critical point of $f_P$, it is also a critical point of $f_{P_1}$.


\end{proof}



Let now $f: (\C^2,0) \to (\C,0)$ be a mixed polynomial, $P=(p_1,p_2)$ a positive weight vector and $\alpha\in\C^*$. We define $d_1:=\min\{d(P;f_u),d(P;f_{\bar{u}})\}$, $d_2:=\min\{d(P;f_v),d(P;f_{\bar{v}})\}$ and
\begin{align}\label{eq:wtilde}
\widetilde{w_1}(u,v)&:=\rmi(\alpha(\overline{f_u})_P-\bar{\alpha}(f_{\bar{u}})_P)_{(d_1,P)}(u,v),\nonumber\\
\widetilde{w_2}(u,v)&:=\rmi(\alpha(\overline{f_v})_P-\bar{\alpha}(f_{\bar{v}})_P)_{(d_2,P)}(u,v).
\end{align}
As is the case for $(w_1,w_2)$ we could have defined $(\widetilde{w_1},\widetilde{w_2})$ equivalently using $f_x$, $x\in\{u,\bar{u},v,\bar{v}\}$, instead of $(f_x)_P$, since the degrees $d_1$ and $d_2$ can by definition only be obtained by $(f_u)_P$ or $(f_{\bar{u}})_P$ and $(f_v)_P$ or $(f_{\bar{v}})_P$, respectively.

There are many similarities between $(w_1,w_2)$ and $(\widetilde{w_1},\widetilde{w_2})$. In particular, by \cite[Lemma 3.5]{AraujoBodeSanchez} we have $d_1\geq d-p_1$ and $d_2\geq d-p_2$ with equalities if and only if $f_P$ depends on $u$ or $\bar{u}$ and $v$ or $\bar{v}$, respectively. It follows that if $f_P$ depends on $v$ or $\bar{v}$, then $\widetilde{w_2}=w_2$. If $f_P$ depends on $u$ or $\bar{u}$, then $\widetilde{w_1}=w_1$. Furthermore, it follows that if $w_i(u_*,v_*)\neq0$ for some $i\in\{1,2\}$ and some $(u_*,v_*)\in\C^2$, then $d_i=d-p_i$ and therefore $\widetilde{w_i}(u_*,v_*)=w_i(u_*,v_*)\neq 0$. The contrapositive of this statement implies the following lemma.

\begin{lemma}\label{lem:ww}
Let $f : (\C^2,0) \to (\C,0)$ be a mixed polynomial, $P$ be a positive weight vector and $\alpha\in\C^*$. Let $(\widetilde{w_1},\widetilde{w_2})$ be as above. Then common zeros of $\widetilde{w_1}$ and $\widetilde{w_2}$ are critical points of $f_P$.
\end{lemma}
\begin{proof}
We know from the above that a common zero $(u_*,v_*)$ of $\widetilde{w_1}$ and $\widetilde{w_2}$ is also a common zero of $w_1$ and $w_2$. But then Lemma~\ref{lem:w} implies that $(u_*,v_*)$ is a critical point of $f_P$.
\end{proof}

The same arguments imply the analogue of Lemma~\ref{lem:conv}. In general, $w_1$ and $w_2$ are easier to handle than $\widetilde{w_1}$ and $\widetilde{w_2}$, because the case when $f_P$ depends neither on $x$ nor on $\bar{x}$ for $x\in\{u,\bar{u},v,\bar{v}\}$ is simpler. However, an advantage of working with $(\widetilde{w_1},\widetilde{w_2})$ is that the assumption on $(1,n)$ in Lemma~\ref{lem:conv} is no longer necessary for the corresponding result on $\widetilde{w_1}$ and $\widetilde{w_2}$.

\begin{lemma}\label{lem:conv2}
Let $f: (\C^2,0) \to (\C,0)$ be a $v$-convenient mixed polynomial, let $P$ be a positive weight vector with $\tfrac{p_1}{p_2}>k_1$ and $\alpha\in\C^*$. Let $(\widetilde{w_1},\widetilde{w_2})$ be as in Eq.~\eqref{eq:wtilde} and $(u_*,v_*)\in \C^2$ be such that $\widetilde{w_1}(u_*,v_*)=\widetilde{w_2}(u_*,v_*)=0$. Then $(0,v_*)$ is a critical point of $f_{P_1}$. Likewise, $(u_*,0)$ is a critical point of $f_{P_N}$ if $f$ is $u$-convenient and $k_N>\tfrac{p_1}{p_2}$.
\end{lemma}
\begin{proof}
The case where $(1,n)\notin\Delta(P_1;f)\cap \text{supp}(f)$ for all $n\in\mathbb{N}_0$ follows directly from the fact that zeros of $(\widetilde{w_1},\widetilde{w_2})$ are also zeros of $(w_1,w_2)$ and Lemma~\ref{lem:conv}.

Now suppose that there is a $n\in\mathbb{N}_0$ with $(1,n)\in\Delta(P_1;f)\cap \text{supp}(f)$. Then $f_{P_1}$ can be written as $f_{P_1}(u,\bar{u},v,\bar{v})=A(v,\bar{v})+uB(v,\bar{v})+\bar{u}C(v,\bar{v})+D(u,\bar{u},v,\bar{v})$ for polynomials $A$, $B$, $C$ and $D$, where each monomial in $D$ is divisible by $u^2$, $u\bar{u}$ or $\bar{u}^2$ and $B$ and $C$ are not both 0. Furthermore, $A$ is not constant 0, because by assumption $f$ is $v$-convenient.

A direct calculation gives $\widetilde{w_2}(u,v)=\rmi(\alpha \overline{A_v}(v,\bar{v})-\bar{\alpha}A_{\bar{v}}(v,\bar{v}))$ and $\widetilde{w_1}(u,v)=\rmi(\alpha\overline{B}(v,\bar{v})-\bar{\alpha}C(v,\bar{v}))$. Note that these expressions do not depend on $u$, so that if $(u_*,v_*)\in\C^2$ is a common zero of $\widetilde{w_1}$ and $\widetilde{w_2}$, so is $(u,v_*)$ for any $u\in \C$. Since $(f_{P_1})_v(0,v)=A_v(v,\bar{v})$ and $(f_{P_1})_{\bar{v}}(0,v)=A_{\bar{v}}(0,v)$ for all $v\in \C$, the equation $\widetilde{w_2}(0,v_*)=0$ implies that $s_{3,f_{P_1}}(0,v_*)=0$. Similarly, we have $(f_{P_1})_u(0,v)=B(v,\bar{v})$ and $(f_{P_1})_{\bar{u}}=C(v,\bar{v})$, so that $\widetilde{w_1}(0,v_*)=0$ implies $s_{2,f_{P_1}}(0,v_*)=0$. It follows from an equation analogous to Eq.~\eqref{eq:frac} that $s_{1,f_{P_1}}(0,v_*)=0$ as well. Therefore, $(0,v_*)$ is a critical point of $f_{P_1}$.
\end{proof}

Let $z(\tau)=(u(\tau),v(\tau))$, $0\leq\tau\leq 1$, be a real-analytic curve in $\C^2$ with $z(0)=(0,0)$ and let $f:(\C^2,0)\to(\C,0)$ be a mixed polynomial. We denote the lowest order of $\tau$ in $f(z(\tau))$ by $d_z(f)$. We write $a_z(f)$ for the coefficient of $\tau^{d_z(f)}$ in $f(z(\tau))$.

The following lemmas are fairly elementary facts on power series, but will be used frequently.

\begin{lemma}\label{lem:lowestorder2}
Let $z(\tau)=(u(\tau),v(\tau))$, $0\leq\tau\leq 1$, be a real-analytic curve in $\C^2$ with $z(0)=(0,0)$ and let $f,g:(\C^2,0)\to(\C,0)$ be mixed polynomials.
\begin{itemize}
\item $d_z(fg)=d_z(f)+d_z(g)$ and $a_z(fg)=a_z(f)a_z(g)$.
\item $d_z(f+g)\geq \min\{d_z(f),d_z(g)\}$ with strict inequality if and only if $d_z(f)=d_z(g)$ and $a_z(f)=-a_z(g)$. If $d_z(f)<d_z(g)$, then $a_z(f+g)=a_z(f)$, and if $d_z(f)=d_z(g)$ with $a_z(f)\neq -a_z(g)$, then $a_z(f+g)=a_z(f)+a_z(g)$. 
\end{itemize}
\end{lemma}


\begin{lemma}\label{lem:lowestorder}
Let $z(\tau)=(u(\tau),v(\tau))$, $0\leq\tau\leq 1$, be a real-analytic curve in $\C^2$ with $z(0)=(0,0)$ and 
\begin{align}
u(\tau)&=a\tau^{p_1}+h.o.t.,\\
v(\tau)&=b\tau^{p_2}+h.o.t.,
\end{align}
for some $a,b\in\C^*$ and $p_1,p_2\in\mathbb{N}$. The expression h.o.t. refers to higher order terms in $\tau$. Let $f: (\C^2,0) \to (\C,0)$ be a mixed polynomial and $P=(p_1,p_2)$. If $f(z(\tau))\not\equiv0$, then the lowest order of $\tau$ in $f(z(\tau))$ is at least $d(P;f)$.
\end{lemma}
\begin{proof}
A direct calculation gives
\begin{equation}
f(z(\tau))=\sum_{d\in \mathbb{N}}f_{(d,P)}(z(\tau))=\sum_{d\in\mathbb{N}}(f_{(d,P)}(a,b)\tau^d+h.o.t),
\end{equation}
where $h.o.t.$~in each summand represents terms whose degree in $\tau$ is greater than $d$.
Thus the lowest order term possible is obtained from the lowest order term of the summand where $d=d(P;f)$, i.e., $f_P(a,b)\tau^{d(P;f)}$. However, since $f_P(a,b)$ could be zero, the actual lowest order $d_z(f)$ could also be greater.
\end{proof}

\begin{proof}[Proof of Theorem~\ref{SMCteo}]
To prove that $U\cap M_{\arg f}=\emptyset $,  for sufficiently small neighbourhoods $U$ of the origin $0\in\mathbb{R}^4$, we proceed as in the proof of \cite[Lemma 11]{Oka2010}. Assume that $f$ does not satisfy the strong Milnor condition. By the curve selection lemma we can find a real analytic curve  $z(\tau)=(u(\tau),v(\tau)), 0\leq \tau \leq 1$, satisfying
\begin{itemize}
    \item[(i)] $z(0)=0$ and $z(\tau)\in \C^2\setminus V_f$ for $\tau>0$. 
    \item[(ii)]$\rmi(f \overline{df} - \bar{f}\bar{d}f)(z(\tau)) = \lambda(\tau)z(\tau)$ for some real number $\lambda(\tau)$. 
\end{itemize}
Since the zeros of $\rmi(f \overline{df} - \bar{f}\bar{d}f)$ are exactly $\Sigma_f\cup V_f$, it does not vanish outside of $V_f$ and close to the origin by Proposition~\ref{strong-isolated3}, which implies that $\lambda(\tau)\not \equiv 0$.\\

\noindent
\underline{Case 1: $ u(\tau)\not \equiv 0$ and $ v(\tau)\not \equiv 0$}\\
Since $z(\tau)$ is real analytic, it is given by a Taylor Series:
$$u(\tau)=a\tau^{p_1}+h.o.t.,\qquad a\neq 0,\, p_1>0,$$
$$v(\tau)=b\tau^{p_2}+h.o.t.,\qquad b\neq 0,\, p_2>0.$$
Define $P =(p_1,p_2),\ z_0 = (a,b) \in (\C^*)^{2}$ and $d = d(P; f)$. We consider the expansions: 
\begin{align}
f(z(\tau))&=\alpha \tau^q+h.o.t.; &q\geq d=d(P;f), \alpha \neq 0
\\
f_x(z(\tau))&=\beta_x \tau^{q_x}+h.o.t.; &x\in\{u,\bar{u}\},\,  q_x\geq d(P;f_x)\geq d-p_1, \beta_x \neq 0  \\
f_y(z(\tau))&=\beta_y\tau^{q_y}+h.o.t.; &y\in\{v,\bar{v}\},\,  q_y\geq d(P;f_y)\geq d-p_2, \beta_y \neq 0  \\
\lambda(\tau)&=\lambda_0\tau^s+h.o.t., &\lambda_0 \in \R, \,s\in\mathbb{Z}_{\geq 0}.
\end{align}
Note that $q_z=d(P;f_z)$ if and only if $(f_z)_P(a,b)\neq 0$ and $\beta_z=(f_z)_P(a,b)$, $z\in \{u,v,\bar{u},\bar{v}\}$.

If Assumption (ii) is satisfied, then that means in particular that the lowest order terms of both functions, $\rmi(f \overline{df} - \bar{f}\bar{d}f)(z(\tau))$ and $\lambda(\tau)z(\tau)$, must agree.

We know from Lemma~\ref{lem:lowestorder2} that the lowest order of the left hand is at least $\min\{d_z(f\overline{f_u}),d_z(\bar{f}f_{\bar{u}})\}$ for the first coordinate and at least $\min\{d_z(f\overline{f_v}),d_z(\bar{f}f_{\bar{v}})\}$ for the second coordinate. Using again Lemma~\ref{lem:lowestorder2} and the fact that $d_z(f)=d_z(\bar{f})=q$, we have at least $q+\min\{d_z(\overline{f_u}),d_z(f_{\bar{u}})\}$ and $q+\min\{d_z(\overline{f_v}),d_z(f_{\bar{v}})\}$, respectively. By Lemma~\ref{lem:lowestorder} these are bounded from below by $q+d_1$ and $q+d_2$, respectively, which by \cite[Lemma 3.5]{AraujoBodeSanchez} are at least $q+d-p_1$ and $q+d-p_2$, respectively.

Therefore, the coefficient of $\tau^{q+d-p_1}$ in $\rmi(f \overline{f_u} - \bar{f}f_{\bar{u}})(z(\tau))$ is either 0 or the lowest order coefficient and thus by Assumption (ii) equal to $\lambda_0a$. Likewise, the coefficient of $\tau^{q+d-p_2}$ in $\rmi(f \overline{f_v} - \bar{f}f_{\bar{v}})(z(\tau))$ is either zero or equal to $\lambda_0b$.



Using again Lemma~\ref{lem:lowestorder2}, Lemma~\ref{lem:lowestorder} and the fact that we already know that $a_z(f)=\alpha$, $d_z(f)=q$, we can calculate these coefficients as $\rmi(\alpha (\overline{f_u})_P - \bar{\alpha}(f_{\bar{u}})_P)_{(d-p_1,P)}(a,b)$ and $\rmi(\alpha (\overline{f_v})_P - \bar{\alpha}(f_{\bar{v}})_P)_{(d-p_2,P)}(a,b)$.


Note that these are exactly the definitions of $w_1(a,b)$ and $w_2(a,b)$ in Lemma~\ref{lem:w}. We claim that $w_1(a,b)$ and $w_2(a,b)$ are both zero.

(Alternatively, we may consider the coefficients of $\tau^{q+d_1}$ and $\tau^{q+d_2}$, which are also either zero or the lowest order coefficients of the left hand side in Assumption (ii). These coefficients are equal to $\widetilde{w_1}(a,b)$ and $\widetilde{w_2}(a,b)$.)

We find that
\begin{align}\label{hermieq1}
&\left \langle \frac{d z(\tau)}{d\tau},\overline{df}(z(\tau))\right \rangle+ \left \langle \frac{d \bar{z}(\tau)}{d\tau},\overline{\bar{d}f}(z(\tau))\right \rangle\nonumber\\
=&\left(f_u(z(\tau))\frac{du(\tau)}{d\tau}+f_{\bar{u}}(z(\tau))\frac{d\bar{u}(\tau)}{d\tau}\right)+\left( f_v(z(\tau))\frac{dv(\tau)}{d\tau}+ f_{\bar{v}}(z(\tau))\frac{d\bar{v}(\tau)}{d\tau}\right)\nonumber\\
=&\frac{d}{d\tau}f(z(\tau))=q\alpha \tau^{q-1}+h.o.t.
\end{align}
where $\langle\, \cdot\,, \,\cdot \,\rangle$ denotes the Hermitian inner product in $\C^2$. By Lemma~\ref{lem:lowestorder} and \cite[Lemma 3.5]{AraujoBodeSanchez} the lowest orders of the left hand side are at least $d-1$ with
\begin{align}
    \left \langle \frac{d z(\tau)}{d\tau},\overline{df}(z(\tau))\right \rangle=&\langle (p_1a,p_2b), ((\overline{f_u})_{(d-p_1,P)}(a,b),(\overline{f_v})_{(d-p_2,P)}(a,b)) \rangle  \tau^{d-1}+h.o.t.\label{hermieq2}\\
 \left \langle \frac{d \bar{z}(\tau)}{d\tau},\overline{\bar{d}f}(z(\tau))\right \rangle=&\langle (p_1\bar{a},p_2\bar{b}), ((\overline{f_{\bar{u}}})_{(d-p_1,P)}(a,b),(\overline{f_{\bar{v}}})_{(d-p_2,P)}(a,b)) \rangle  \tau^{d-1}+h.o.t.,\label{hermieq3}  
\end{align}
Recall from \cite[Lemma 3.5]{AraujoBodeSanchez} that for all $x\in\{u,\bar{u},v,\bar{v}\}$ we have either $f_x=0$ or $d(P;f_x)=d-p_i$, where $i=1$ if $x\in\{u,\bar{u}\}$ and $i=2$ if $x\in\{v,\bar{v}\}$. We may thus calculate the coefficient of $\tau^{d-1}$ in $\left(\left \langle \frac{d z(\tau)}{d\tau},\overline{df}(z(\tau))\right \rangle+ \left \langle \frac{d \bar{z}(\tau)}{d\tau},\overline{\bar{d}f}(z(\tau))\right \rangle\right)$ and get either zero or $q\alpha$ (in which case $d=q$).

Note that 
 \begin{align}
 &\langle (p_1a,p_2b), ((\overline{f_u})_{(d-p_1,P)}(a,b),(\overline{f_v})_{(d-p_2,P)}(a,b))\rangle \nonumber\\
 &+ \langle (p_1\bar{a},p_2\bar{b}), (\overline{f_{\bar{u}}}_{(d-p_1,P)}(a,b),(\overline{f_{\bar{v}}})_{(d-p_2,P)}(a,b))\rangle\nonumber\\
 =&  (-\rmi \bar{\alpha})^{-1}\left(  \langle (p_1a,p_2b), (\rmi \alpha (\overline{f_u})_{(d-p_1,P)}(a,b),\rmi \alpha (\overline{f_v})_{(d-p_2,P)}(a,b))\rangle\right.\nonumber\\
 &\left.+ \langle (p_1\bar{a},p_2\bar{b}), (\rmi\alpha\overline{f_{\bar{u}}}_{(d-p_1,P)}(a,b),\rmi \alpha (\overline{f_{\bar{v}}})_{(d-p_2,P)}(a,b))\rangle \right)\label{thm41eq2}
\end{align}
and
\begin{align}\label{thm41eq1}
   &\operatorname{Re} \left( \langle (p_1a,p_2b), (\rmi \alpha (\overline{f_u})_{(d-p_1,P)}(a,b),\rmi \alpha (\overline{f_v})_{(d-p_2,P)}(a,b))\rangle\right.\nonumber\\
   &\left.+  \langle (p_1\bar{a},p_2\bar{b}), (\rmi\alpha\overline{f_{\bar{u}}}_{(d-p_1,P)}(a,b),\rmi \alpha (\overline{f_{\bar{v}}})_{(d-p_2,P)}(a,b))\rangle  \right)\nonumber\\
   =& \operatorname{Re} \left( \langle (p_1a,p_2b), (\rmi \alpha (\overline{f_u})_{(d-p_1,P)}(a,b),\rmi \alpha (\overline{f_v})_{(d-p_2,P)}(a,b))\rangle \right)\nonumber\\
   &+ \operatorname{Re} \left( \langle (p_1\bar{a},p_2\bar{b}), (\rmi\alpha\overline{f_{\bar{u}}}_{(d-p_1,P)}(a,b),\rmi \alpha (\overline{f_{\bar{v}}})_{(d-p_2,P)}(a,b))\rangle  \right)\nonumber\\
   =&  \operatorname{Re} \left( \langle (p_1a,p_2b), (\rmi \alpha (\overline{f_u})_{(d-p_1,P)}(a,b),\rmi \alpha (\overline{f_v})_{(d-p_2,P)}(a,b))\rangle \right)\nonumber\\
   &+ \operatorname{Re} \left( \langle (p_1a,p_2b), (-\rmi\bar{\alpha}(f_{\bar{u}})_{(d-p_1,P)}(a,b),-\rmi \bar{\alpha} (f_{\bar{v}})_{(d-p_2,P)}(a,b))\rangle  \right)\nonumber\\
   =&   \operatorname{Re} \left( \langle (p_1a,p_2b), (\rmi \alpha (\overline{f_u})_{(d-p_1,P)}(a,b)-\rmi\bar{\alpha}(f_{\bar{u}})_{(d-p_1,P)}(a,b),\right.\nonumber\\
   &\left.\rmi \alpha  (\overline{f_v})_{(d-p_2,P)}(a,b)-\rmi \bar{\alpha}(f_{\bar{v}})_{(d-p_2,P)}(a,b) )\rangle  \right)\nonumber\\
   =&  \operatorname{Re} \left( \langle (p_1a,p_2b), (w_1(a,b), w_2(a,b)) \rangle \right). 
\end{align}

Suppose that $w_1(a,b)$ and $w_2(a,b)$ are not both equal to 0.
Recall that $w_1(a,b)$ and $w_2(a,b)$ are equal to $\lambda_0a$ and $\lambda_0b$, respectively, if they are non-zero. From this we obtain that $\langle (p_1a,p_2b), (w_1(a,b), w_2(a,b)) \rangle$ is equal to $\lambda_0|a|^2p_1$, $\lambda_0|b|^2p_2$ or $\lambda_0(|a|^2p_1+|b|^2p_2)$, depending on which of $w_1(a,b)$ and $w_2(a,b)$ vanishes. In either case, we obtain a non-zero real number if $w_1(a,b)\neq 0$ or $w_2(a,b)\neq 0$, i.e., 
\begin{equation}\label{eq:ineqw1}
\operatorname{Re} \left( \langle (p_1a,p_2b), (w_1(a,b),w_2(a,b)) \rangle \right)\neq 0,
\end{equation} which implies that $d=q$ and
\begin{align}
&\langle (p_1a,p_2b), \left((\overline{f_u})_{(d-p_1,P)}(a,b),(\overline{f_v})_{(d-p_2,P)}(a,b))\rangle\right. \nonumber\\
&\left.+ \langle (p_1\bar{a},p_2\bar{b}), (\overline{f_{\bar{u}}}_{(d-p_1,P)}(a,b),(\overline{f_{\bar{v}}})_{(d-p_2,P)}(a,b)\right)\rangle=q\alpha.
\end{align}
Then Eq.~\eqref{thm41eq2} implies that
\begin{align}
&\langle (p_1a,p_2b), \left(\rmi \alpha (\overline{f_u})_{(d-p_1,P)}(a,b),\rmi \alpha (\overline{f_v})_{(d-p_2,P)}(a,b))\rangle\right.\nonumber\\
&\left.+ \langle (p_1\bar{a},p_2\bar{b}), (\rmi\alpha\overline{f_{\bar{u}}}_{(d-p_1,P)}(a,b),\rmi \alpha (\overline{f_{\bar{v}}})_{(d-p_2,P)}(a,b)\right)\rangle =-\rmi q|\alpha|^2.
\end{align}


Taking the real part on both sides, we get an obvious contradiction from Eqs.~\eqref{thm41eq1} and \eqref{eq:ineqw1}:
\begin{equation}\label{contradict1}
0=\operatorname{Re} \left( \langle (p_1a,p_2b), (w_1(a,b),w_2(a,b)) \rangle \right)\neq 0.
\end{equation}

This proves the claim that $w_1(a,b)$ and $w_2(a,b)$ are both 0.


If $k_1\geq \tfrac{p_1}{p_2}\geq k_N$, then $\Delta(P)$ is not an extreme vertex. Then by Lemma~\ref{lem:w} $(a,b)\in(\C^*)^2$ is a critical point of $f_P$, contradicting SIND-(ii).

If $\tfrac{p_1}{p_2}>k_1$ and $f$ is not $v$-convenient, then $f_P=f_{\Delta_1}$ is SND by Lemma~\ref{lem:extremeSND}. Again, $(a,b)\in(\C^*)^2$ is a critical point of $f_P$, which is a contradiction to $f_P$ being strongly non-degenerate. The case of $\tfrac{p_1}{p_2}<k_N$ and $f$ not $u$-convenient is analogous.

If $\tfrac{p_1}{p_2}>k_1$, $f$ is $v$-convenient and $(1,n)\notin\Delta(P_1)\cap\text{supp}(f)$ for all $n\in\mathbb{N}_0$, then by Lemma~\ref{lem:conv} $(0,b)$ is a critical point of $f_{P_1}$, contradicting SIND-(i). The analogous statement holds if $\tfrac{p_1}{p_2}<k_N$, $f$ is $u$-convenient and $(n,1)\notin\Delta(P_N)\cap\text{supp}(f)$ for all $n\in\mathbb{N}_0$.

We are left with one last case (within Case 1), which is that $\tfrac{p_1}{p_2}>k_1$, $f$ is $v$-convenient and there is an $n\in\mathbb{N}_0$ with $(1,n)\in\Delta(P_1)\cap\text{supp}(f)$ (and the analogous case where $\tfrac{p_1}{p_2}<k_N$, $f$ is $u$-convenient and $(n,1)\in\Delta(P_N)\cap\text{supp}(f)$ for some $n\in\mathbb{N}_0$).

Since $\tfrac{p_1}{p_2}>k_1$, we know that $f_P=f_{\Delta_1}$ and it depends on $v$ or $\bar{v}$, so that $\widetilde{w_2}=w_2$ and therefore $\widetilde{w_2}(a,b)=w_2(a,b)=0$ by Eqs.~\eqref{thm41eq1}-\eqref{contradict1}. Now if $\widetilde{w_1}(a,b)$ is 0 as well, then Lemma~\ref{lem:conv2} implies that $(0,b)$ is a critical point of $f_{P_1}$, contradicting SIND-(i). Thus $\widetilde{w_1}(a,b)\neq 0$. Note that in this case $\widetilde{w_1}(a,b)$ is the coefficient of $\tau^{q+d_1}$ in $\rmi(f \overline{f_u} - \bar{f}f_{\bar{u}})(z(\tau))$ and by Lemma~\ref{lem:lowestorder} it is the lowest order coefficient. Therefore, Assumption (ii) implies that $q+d_1=s+p_1$.

On the other hand, $\widetilde{w_2}(a,b)$, the coefficient of $\tau^{q+d_2}$ of the second coordinate of the left hand side in Assumption (ii), vanishes. This means (again by Lemma~\ref{lem:lowestorder}) that the lowest order of $\tau$ in the second coordinate of the right hand side is strictly greater than $q+d_2$. Comparing with the lowest orders of the right hand side, we obtain
\begin{align}\label{eq:degs}
q+d_1&=s+p_1\nonumber\\
q+d_2&<s+p_2.
\end{align} 

Now note that if there is an $n\in\mathbb{N}_0$ with $(1,n)\in \Delta(P_1)\cap\text{supp}(f)$, then $k_1\geq 1$. Otherwise, the intersection between $\Delta(P_1)$ and the vertical line through $(1,n)$ would not be an integer lattice point. Since $f$ is $v$-convenient, there is an intersection between $\Delta(P_1)$ and the vertical axis. Denote this intersection point by $(0,n_0)$ with $n_0\in\mathbb{N}$. Then $n=n_0-k_1$ and since $k_1\geq 1$, we have $n_0-n\geq 1$. Note that $d_2=(n_0-1)p_2$ and $d_1=np_2$. It follows from Eq.~\eqref{eq:degs} that
\begin{align}
q+np_2=s+p_1,\\
q+(n_0-1)p_2&<s+p_2
\end{align}
and therefore
\begin{align}
&q+(n_0-1)p_2<q+np_2-p_1+_2\nonumber\\
\implies &(n_0-n-1)p_2<p_2-p_1\nonumber\\
\implies &0\leq n_0-n-1<1-\frac{p_1}{p_2}<1-k_1<0,
\end{align}
which is a contradiction and concludes our discussion of Case 1.\\

\noindent
\underline{Case 2: $u(\tau)\equiv 0$ or $v(\tau)\equiv 0$}\\
We discuss the case of $u(\tau)\equiv 0$. The case of $v(\tau)\equiv 0$ is analogous.

Since $u(\tau)$ and $v(\tau)$ cannot both be constant 0, we have $v(\tau)\not\equiv0$. It is thus given by a Taylor series $v(\tau)=b\tau^{p_2}+h.o.t.$ with $b\in\C^*$. We may now pick any positive weight vector $P$ of the form $P=(p_1,p_2)$ with $\tfrac{p_1}{p_2}>k_1$, that is, we are free to choose $p_1$ sufficiently large, while $p_2$ is determined by the Taylor series of $v(\tau)$. Since $f(z(\tau))\neq 0$ for $\tau>0$ by Assumption (i) and $u(\tau)\equiv 0$, it follows that $f$ is $v$-convenient. Otherwise, $f(0,v)=0$ for all $v\in \mathbb{C}$, contradicting Assumption (i). This implies that as in Case 1, $\widetilde{w_1}(0,b)$ and $\widetilde{w_2}(0,b)$ cannot both be 0. Otherwise, $(0,b)$ would be a critical point of $f_{P_1}$ by Lemma~\ref{lem:conv2}, contradicting SIND-(i). Since $\widetilde{w_1}(0,b)$ is the coefficient of $\tau^{q+d_1}$ in the first coordinate of the left hand side of the equation in Assumption (ii), it must vanish, since the first coordinate of the right hand side is constant 0. Since $\widetilde{w_1}(0,b)=0$, we must therefore have $\widetilde{w_2}(0,b)\neq 0$. Since $\tfrac{p_1}{p_2}>k_1$, $f_P=f_{\Delta_1}$ and it depends on $v$ or $\bar{v}$, which implies that $\widetilde{w_2}=w_2$ and thus $w_2(0,b)\neq 0$. But then the same calculation as in Case 1 results in
\begin{equation}
0=\operatorname{Re} \left( \langle (p_10,p_2b), (w_1(0,b),w_2(0,b)) \rangle \right)=\lambda_0|b|^2p_2\neq 0.
\end{equation}

We thus obtain a contradiction in both cases, which proves that $f$ satisfies the strong Milnor condition.
\end{proof}


Recall from Corollary~\ref{radialequivweak} that for a radially weighted homogeneous mixed polynomial $f$, the presence of a weakly isolated singularity and the properties PND and IND are all equivalent, and also in the strong sense, i.e., an isolated singularity, SPND and SIND are all equivalent. By Theorem~\ref{SMCteo} and Corollary~\ref{radialequivweak}, a radially weighted homogeneous mixed polynomial $f$ with isolated singularity satisfies $\Sigma_{\arg(f)}=M_{\arg (f)}=\emptyset$. (Note that for radially weighted homogeneous polynomials $M_{\arg (f)}=\emptyset$ and $\Sigma_{\arg(f)}=\emptyset$ if and only if $U\cap M_{\arg(f)}=\emptyset$ and $U\cap \Sigma_{\arg(f)}=\emptyset$, respectively, for some open neighbourhood $U$ of the origin.) It is known that if $f$ is also semiholomorphic, then $\Sigma_f \setminus V_f = \Sigma_{\arg(f)}$. This implies that for semiholomorphic radially weighted homogeneous polynomials, under the assumption of a weakly isolated singularity, the strong Milnor condition is equivalent to the existence of an isolated singularity at the origin. We will see now that this equivalence is also true for all radially weighted homogeneous mixed polynomials, not only semiholomorphic ones.
\begin{prop}\label{SMCprop}
Let $f$ be a radially weighted homogeneous mixed polynomial with weakly isolated singularity. Then $M_{\arg{f}}=\Sigma_{\arg f}=\Sigma_f \setminus \{0\}$. Moreover, $f$ satisfies the strong Milnor condition if and only if $f$ has an isolated singularity at the origin.   
\end{prop}   
\begin{proof}
From Theorem~\ref{strong-isolated3} we have that $M_{\arg{f}}=\Sigma_{\arg f}$.  The fact that $f$ has a weakly isolated singularity implies that $\Sigma_f \setminus V_f = \Sigma_f \setminus\{0\}$.  Thus it suffices to prove that $\Sigma_f\setminus \{0\}=\Sigma_{\arg f}$. From a direct calculation using the definitions of both sets we have that $\Sigma_{\arg f} \subseteq \Sigma_f$. Consider the real mapping $(|f|,\arg f)$, and let $p=(u_*,v_*)$ a regular point of $\arg f$ . Then there exist 3 directions (linear independent) $w_1,w_2, \text{ and } w_3 \in T_p \R^4$ such that $\frac{\partial \arg f }{dw_i}=0, \ i=1,2,3$, and some direction $w_4$ such that $\frac{\partial \arg f }{dw_4}(p)\neq 0$.  On the other hand, note that for $T=(p_1\text{Re}(u),p_1\text{Im}(u),p_2 \text{Re}(v),p_2\text{Im}(v))\in \mathbb{R}^4$, where $P=(p_1,p_2)$ is the radial weight type of $f$, we have by the radial action that $T\cdot\nabla |f|(p)\neq 0$ and also that $T\cdot\nabla \arg f(p)= 0$. It follows that $T$, interpreted as an element of $T_p\mathbb{R}^4$, $p=(\text{Re}(u),\text{Im}(u),\text{Re}(v),\text{Im}(v))$, is in the span of $w_1,w_2$ and $w_3$ and we may (after redefining $w_1$ and $w_2$) assume that $w_1, w_2$ and $T$ span the tangent space of the level set of $\arg(f)$ at $p$. Therefore, the real Jacobian in the basis $(w_1,w_2,T,w_4)$ is 
\begin{align}
J f (p)=
\begin{pmatrix}
\small
 \frac{\partial |f| }{d w_1}(p)  &  \frac{\partial |f| }{d w_2}(p)  & \frac{\partial |f| }{d T}(p)\neq 0  &  \frac{\partial |f| }{d w_4}(p)  \\ \\
0 &0& \frac{\partial \arg f }{d T}(p)=0 &  \frac{\partial \arg f }{d w_4}(p)\neq 0 
\end{pmatrix}
\end{align}
Thus, $p$ is a regular point of $f$, which implies $\Sigma_f \setminus V_f \subseteq \Sigma_{\arg f}$. Hence, under the assumption that $f$ has a weakly isolated singularity, so that $\Sigma_f\cap V_f=\{0\}$, we have $M_{\arg f }=\emptyset$ if and only if $\Sigma_f = \{0\}$. Since the strong Milnor condition is equivalent to $M_{\arg f }=\emptyset$  in the case of weakly isolated singularity, we have proved the proposition. 
\end{proof}
A different proof of $M_{\arg{f}}=\Sigma_{\arg f}=\Sigma_f \setminus \{0\}$ for any radially weighed homogeneous mixed polynomial $f$ can be found in \cite{Chen2014}, where the author uses a different technique. Together with Corollary~\ref{radialequivweak} the proposition proves Theorem~\ref{teo:rad}. 


  
  \bibliographystyle{amsplain}
  \renewcommand{\refname}{ \large R\normalsize  EFERENCES}
\bibliography{sample-2}
\Addresses
\end{document}